\documentclass[12pt]{article}
\usepackage{amssymb,bm}
\usepackage{amsmath}
\usepackage{amsthm}
\usepackage{graphicx}
\usepackage{hyperref} 
\hypersetup{pdfborder=0 0 0}
\newtheorem{theorem}{{\sc Theorem}}[section]

\newtheorem{lemma}[theorem]{{\sc Lemma}}
\newtheorem{corollary}[theorem]{Corollary}
\newtheorem{remark}[theorem]{Remark}

\newtheorem{conjecture}[theorem]{Conjecture}

\newcommand{\bb}[1]{\mathbb{ #1}}

\bmdefine\Bone{1}

\newcommand{\weak}{\rightharpoonup\:}

\newcommand{\bra}[1]{\overline{#1}}

\newcommand{\hf}{\displaystyle\frac{1}{2}}
\newcommand{\nth}[1]{\displaystyle\frac{1}{#1}}

\newcommand{\Md}{\partial}

\def\XXint#1#2#3{{\setbox0=\hbox{$#1{#2#3}{\int}$ }
\vcenter{\hbox{$#2#3$ }}\kern-.6\wd0}}

\newcommand{\jump}[1]{\lbrack\!\lbrack #1 \rbrack\!\rbrack}

\newcommand{\lims}{\mathop{\overline\lim}}
\newcommand{\limi}{\mathop{\underline\lim}}

\newcommand{\sech}{\mathop{\mathrm{sech}}}

\newcommand{\rhs}{right-hand side}
\newcommand{\lhs}{left-hand side}

\newcommand{\WLOG}{without loss of generality}

\newcommand{\IFF}{if and only if }

\newcommand{\Ga}{\alpha}
\newcommand{\Gb}{\beta}
\newcommand{\Gd}{\delta}
\newcommand{\Ge}{\epsilon}

\newcommand{\Gg}{\gamma}

\newcommand{\Gl}{\lambda}

\newcommand{\Gth}{\theta}

\newcommand{\Gz}{\zeta}

\newcommand{\GG}{\Gamma}

\newcommand{\GO}{\Omega}

\bmdefine\BGa{\alpha}
\bmdefine\BGb{\beta}
\bmdefine\BGd{\delta}
\bmdefine\BGe{\epsilon}
\bmdefine\BGve{\varepsilon}
\bmdefine\BGf{\phi}
\bmdefine\BGvf{\varphi}
\bmdefine\BGg{\gamma}
\bmdefine\BGc{\chi}
\bmdefine\BGi{\iota}
\bmdefine\BGk{\kappa}
\bmdefine\BGl{\lambda}
\bmdefine\BGn{\eta}
\bmdefine\BGm{\mu}
\bmdefine\BGv{\nu}
\bmdefine\BGp{\pi}
\bmdefine\BGth{\theta}
\bmdefine\BGvth{\vartheta}
\bmdefine\BGr{\rho}
\bmdefine\BGvr{\varrho}
\bmdefine\BGs{\sigma}
\bmdefine\BGvs{\varsigma}
\bmdefine\BGt{\tau}
\bmdefine\BGj{\tau}
\bmdefine\BGu{\upsilon}
\bmdefine\BGo{\omega}
\bmdefine\BGx{\xi}
\bmdefine\BGy{\psi}
\bmdefine\BGz{\zeta}
\bmdefine\BGD{\Delta}
\bmdefine\BGF{\Phi}
\bmdefine\BGG{\Gamma}
\bmdefine\BGL{\Lambda}
\bmdefine\BGP{\Pi}
\bmdefine\BGT{\Theta}
\bmdefine\BGS{\Sigma}
\bmdefine\BGU{\Upsilon}
\bmdefine\BGO{\Omega}
\bmdefine\BGX{\Xi}
\bmdefine\BGY{\Psi}

\bmdefine\BFM{\mathfrak{M}}
\bmdefine\BFb{\mathfrak{b}}
\bmdefine\BFk{\mathfrak{k}}
\bmdefine\BFm{\mathfrak{m}}
\bmdefine\BFu{\mathfrak{u}}
\bmdefine\BFv{\mathfrak{v}}

\newcommand{\CK}{{\mathcal K}}

\newcommand{\CM}{{\mathcal M}}

\newcommand{\CR}{{\mathcal R}}

\bmdefine\BCA{{\mathcal A}}
\bmdefine\BCB{{\mathcal B}}
\bmdefine\BCC{{\mathcal C}}
\bmdefine\BCD{{\mathcal D}}
\bmdefine\BCE{{\mathcal E}}
\bmdefine\BCF{{\mathcal F}}
\bmdefine\BCG{{\mathcal G}}
\bmdefine\BCH{{\mathcal H}}
\bmdefine\BCI{{\mathcal I}}
\bmdefine\BCJ{{\mathcal J}}
\bmdefine\BCK{{\mathcal K}}
\bmdefine\BCL{{\mathcal L}}
\bmdefine\BCM{{\mathcal M}}
\bmdefine\BCN{{\mathcal N}}
\bmdefine\BCO{{\mathcal O}}
\bmdefine\BCP{{\mathcal P}}
\bmdefine\BCQ{{\mathcal Q}}
\bmdefine\BCR{{\mathcal R}}
\bmdefine\BCS{{\mathcal S}}
\bmdefine\BCT{{\mathcal T}}
\bmdefine\BCU{{\mathcal U}}
\bmdefine\BCV{{\mathcal V}}
\bmdefine\BCW{{\mathcal W}}
\bmdefine\BCX{{\mathcal X}}
\bmdefine\BCY{{\mathcal Y}}
\bmdefine\BCZ{{\mathcal Z}}

\bmdefine\Bzr{ 0}
\bmdefine\Ba{ a}
\bmdefine\Bb{ b}
\bmdefine\Bc{ c}
\bmdefine\Bd{ d}
\bmdefine\Be{ e}
\bmdefine\Bf{ f}
\bmdefine\Bg{ g}
\bmdefine\Bh{ h}
\bmdefine\Bi{ i}
\bmdefine\Bj{ j}
\bmdefine\Bk{ k}
\bmdefine\Bl{ l}
\bmdefine\Bm{ m}
\bmdefine\Bn{ n}
\bmdefine\Bo{ o}
\bmdefine\Bp{ p}
\bmdefine\Bq{ q}
\bmdefine\Br{ r}
\bmdefine\Bs{ s}
\bmdefine\Bt{ t}
\bmdefine\Bu{ u}
\bmdefine\Bv{ v}
\bmdefine\Bw{ w}
\bmdefine\Bx{ x}
\bmdefine\By{ y}
\bmdefine\Bz{ z}
\bmdefine\BA{ A}
\bmdefine\BB{ B}
\bmdefine\BC{ C}
\bmdefine\BD{ D}
\bmdefine\BE{ E}
\bmdefine\BF{ F}
\bmdefine\BG{ G}
\bmdefine\BH{ H}
\bmdefine\BI{ I}
\bmdefine\BJ{ J}
\bmdefine\BK{ K}
\bmdefine\BL{ L}
\bmdefine\BM{ M}
\bmdefine\BN{ N}
\bmdefine\BO{ O}
\bmdefine\BP{ P}
\bmdefine\BQ{ Q}
\bmdefine\BR{ R}
\bmdefine\BS{ S}
\bmdefine\BT{ T}
\bmdefine\BU{ U}
\bmdefine\BV{ V}
\bmdefine\BW{ W}
\bmdefine\BX{ X}
\bmdefine\BY{ Y}
\bmdefine\BZ{ Z}

\DeclareMathAlphabet{\cg}{OMS}{zplm}{m}{n}
\newcommand{\HH}{\mathbb{H}}
\newcommand{\NN}{\mathbb{N}} 
\newcommand{\RR}{\mathbb{R}}      
\newcommand{\D}{\text{d}}
\usepackage{epstopdf}
\usepackage{subcaption}
\title{Optimal error estimates for analytic continuation in the upper half-plane}
\author{Yury Grabovsky \and Narek Hovsepyan}
\begin{document}
\maketitle
\begin{abstract}
  Analytic functions in the Hardy class $H^{2}$ over the upper half-plane
  $\mathbb{H}_+$ are uniquely determined by their values on any curve $\Gamma$
  lying in the interior or on the boundary of $\mathbb{H}_+$. The goal of this
  paper is to provide a sharp quantitative version of this statement. We
  answer the following question. Given $f$ of a unit $H^{2}$ norm that is
  small on $\Gamma$ (say, its $L^{2}$ norm is of order $\epsilon$), how large
  can $f$ be at a point $z$ away from the curve? When $\Gamma \subset \partial
  \mathbb{H}_+$, we give a sharp upper bound on $|f(z)|$ of the form
  $\epsilon^\gamma$, with an explicit exponent $\gamma = \gamma(z) \in (0,1)$
  and explicit maximizer function attaining the upper bound.  When $\Gamma
  \subset \mathbb{H}_+$ we give an implicit sharp upper bound in terms of a
  solution of an integral equation on $\Gamma$. We conjecture and give
  evidence that this bound also behaves like $\epsilon^\gamma$ for some
  $\gamma = \gamma(z) \in (0,1)$. These results can also be transplanted to
  other domains conformally equivalent to the upper half-plane.
\end{abstract}


\section{Introduction}
\setcounter{equation}{0} 
\label{sec:intro}
Our motivation
comes from the effort to understand stability of extrapolation of complex
electromagnetic permittivity of materials as a function of frequency
\cite{lali60:8,feyn64}. An underlying mathematical problem is about
identifying a Herglotz function---a complex analytic function in the upper
half-plane $\bb{H}_{+}$ that has nonnegative imaginary part, given its values
at specific points in the upper half-plane or on its boundary. Such functions,
and their variants, are ubiquitous in physics. For example, the complex
impedance of an electrical circuit as a function of frequency has a similar
property. Yet another example, is the dependence of effective moduli of
composites on the moduli of its constituents \cite{berg78,mi80}. These
functions appear in areas as diverse as optimal design problems
\cite{lipt01a,lipt01}, nuclear physics \cite{capr74,capr79} and medical
imaging \cite{epst08}. It is simply impossible to enumerate all of the fields
in science and engineering where they occur. Notwithstanding a more than a
century of attention, Herglotz functions remain at the forefront of research,
e.g. \cite{dlp17,cami17,sfjm18}.

Let us assume that a Herglotz function has been experimentally measured on a
curve $\GG$ in $\bb{H}_{+}$. The measurements may contain small errors and the
actual data may no longer come from any Herglotz function. The goal is to find
a Herglotz function consistent with such noisy measurements up to a small
error. In this paper we are not interested in any specific reconstruction or
extrapolation algorithm, of which there is an overabundance in the
literature, but rather in characterizing a worst case scenario, where two
Herglotz functions differ little at the data points, but may diverge
significantly, the further away from the data source we move. Since Herglotz
functions that decay at infinity always lie in a Hardy space $H^{2}$ of the
upper half-plane, we will ask how large can a Hardy function, representing the
difference between two Herglotz extrapolants of the same data be
at a specific point $z$ if we know that it is $L^{2}$-small on a curve
$\GG$ in the upper half-plane $\bb{H}_{+}$.

On the one hand complex analytic functions possess a large degree of rigidity,
being uniquely determined by values at any infinite set of points in a compact
set. This rigidity implies that even very small measurement errors will
produce data \emph{mathematically} inconsistent with values of an analytic
function. On the other hand there is a theorem due to Riesz (see
e.g. \cite{part97}) that restrictions of analytic functions in a Hardy class
$H^{2}$ are dense in $L^{2}$ on any smooth bounded curve.  Therefore,
\emph{any data} can be extrapolated as an analytic function with arbitrary
degree of agreement. The high accuracy of matching will be attained by an
increasingly wild behavior away from the curve \cite{digr01}. To see why this
occurs we can examine Carleman formulas \cite{carl26,gokr33} expressing values
of the analytic function in the domain in terms of its values on a part of the
boundary.  These formulas are highly oscillatory and reproduce values of
analytic functions using delicate \emph{exact} cancellation properties such
functions enjoy. Small measurement errors destroy these exact cancellations
and small errors get exponentially amplified. For curves in the interior
Carleman type formulas have been developed in \cite{aize12}, but they exhibit
the same error amplification feature since they are also based on the same
exact cancellation properties of analytic functions.

Typically analytic continuation problems are regularized by placing additional
boundedness constraints on the extrapolant. The resulting competition between
``rigidity'' and ``flexibility'' of complex analytic functions place such
questions between ill and well-posed problems. Our goal is to obtain a
quantitative version of such a statement. We therefore formulate the
\emph{power law transition principle} according to which all so regularized
analytic continuation problems must exhibit a power law transition from well
to ill-posedness.  Specifically, if $f(\Gz)$ is bounded in some norm in the
space of analytic functions on a domain $\GO$, and is also of order $\Ge$ on a
curve $\GG\subset\bra{\GO}$ in some other norm (e.g. in $L^{2}(\GG)$ or
$L^{\infty}(\GG)$), then it can be only as large as $C\Ge^{\Gg(z)}$ at some
other point $z\in\GO\setminus\GG$. Moreover $\Gg(z)\in(0,1)$ decreases from 1
to 0, as the point $z$ moves further and further away from $\GG$. This general
principle in the form of an upper bound has been recently established in
\cite{trefe19}. In fact, upper and lower bounds of this form have long been
known in the literature,
e.g. \cite{cami65,mill70,payne75,fran90,vese99,fdfd09,deto18,trefe19}. However,
exact values of $\Gg(z)$ have only recently been obtained in a few special cases
\cite{deto18,trefe19} by matching bounds and constructions.

The most common regularizing boundedness constraints in the literature are in the
$H^{\infty}(\GO)$ norm. The power law estimates are then derived from a
maximum modulus principle. To take a simple example from \cite{trefe19}, the modulus
of the function $e^{z\ln\Ge}f(z)$ does not exceed $\Ge$ on the boundary of the
infinite strip $\Re z\in(0,1)$, provided $|f(z)|\le 1$ in the strip and
$|f(iy)|\le\Ge$. The maximum modulus principle (or rather its
Phragm\'en-Lindel\"of version) then implies that $|f(z)|\le \Ge^{1-\Re
  z}$. The estimate is optimal, since $f(z)=\Ge e^{-z\ln\Ge}$ satisfies the
constraints and achieves equality in the maximum modulus principle. We believe
that the power law transition principle for analytic continuation holds in a
wide variety of contexts irrespective of the choice of norms, domain geometries and
sources of data.

In this paper we formulate the problem of optimal analytic continuation error
estimates using Hilbert space norms, rather than $H^{\infty}$ norms and use
variational methods that establish optimal upper bounds on the
extrapolation error. The bounds are formulated in terms of the solution of an
integral equation. In this new formulation the power law
transition principle is contained in a somewhat implicit form. It can be made
explicit in those cases where the underlying integral equation can be
solved explicitly, as is done in Section~\ref{extrapolation from boundary SECT}, and in
the companion paper \cite{grho-annulus}. There we apply our methodology
to the setting of \cite{deto18}, but with $L^{2}$ rather than $L^{\infty}$
norms. We recover their power law exponent, suggesting that the exponents must
be robust and not very sensitive to the choice of specific norms in the spaces
of analytic functions. This phenomenon could be related to the fact that
functions with worst extrapolation error can be analytically continued into
much larger domains, as is evident from our integral equation, and hence satisfy
the required constraints in all $L^{p}$ or $H^{p}$ norms.

Conformal mappings between domains can be used to ``transplant'' the exponent
estimates from one geometry to a different one. For example, we can transplant
the exponent obtained in Section~\ref{extrapolation from boundary SECT} for
the half-plane to the half-strip $\Re z>0$, $|\Im z|<1$, considered in
\cite{trefe19}. The analytic function $f(z)$ is assumed to be bounded in the
half-strip and also of order $\Ge$ on the interval $[-i,i]$ on the imaginary
axis. Then any such function must satisfy $|f(x)|\le C\Ge^{\Gg(x)}$, $x>0$,
where $\Gg(x)=(2/\pi){\rm arccot}(\sinh(\pi x/2))$. Moreover, the estimate is
sharp, since it is attained by the function $W(-i\sinh(\pi z/2))$, where
$W(\Gz)$ is given by (\ref{maximizer boundary}). This result follows from the
observation that $\Gz=-i\sinh(\pi z/2)$ is a conformal map from the half-strip
to the upper half-plane, mapping interval $[-i,i]$ to the interval
$[-1,1]$. As Trefethen points out in \cite{trefe19}, the half-strip geometry
gives a stark example of the discrepancy between mathematical well-posedness
(the analytic continuation error does go to 0 as $\Ge\to 0$) and practical
well-posedness: At $x=1$ only a quarter of all digits of precision will
remain, while at $x=2$ only 1/20th will remain.

We start our analysis by reformulating the problem as a maximization of a
linear functional with quadratic inequality constraints, which is why we use
Hilbert space norms in the original problem formulation. We then use convex
duality to obtain an upper bound on $f(z)$. The conditions of optimality of
the bound lead to an integral equation for the worst case function
$u(\Gz)$. We conclude that our upper bound is optimal, since $u(\Gz)$
satisfies all the constraints. We show that the power law transition principle
is a consequence of the conjectured exponential decay of the eigenvalues and
eigenfunctions of the integral operator (see Theorem~\ref{connections
  THM}). The eigenvalues of the integral operator are also singular values of
the restriction operator \cite{gps03}, whose exact exponential decay rates are
well-known in some cases \cite{parf78,parfenov}. The integral operator in the
upper half-plane possesses a special ``displacement structure'', and the
exponential decay of its eigenvalues follows from the upper bound in
\cite{beto17}. Our numerical computations (with the help of Leslie Greengard)
show that this upper bound matches the rate of exponential decay of
eigenvalues extremely well, when $\GG$ is the interval $[-1,1]+ih$, $h>0$. In
this special case the integral operator is also of finite convolution type and
upper and lower bounds on the rate of exponential decay of its eigenvalues
follow from results of Widom \cite{widom64}. Our computations show that these
bounds are far from optimal.
 
The paper is organized as follows. In the next section we state and discuss
our main results. In section~\ref{SEC K norm estimates} we attempt some a
priori analysis of the integral equation generating the maximizer of the
analytic continuation error. This exhibits singular features of
the integral equation that defeat standard a priori estimates approaches. In
section~\ref{conjecture implies optimality SECTION} we show how the power law
transition principle arises from putative features of the integral equation,
such as exponential decay of its eigenvalues.
In section~\ref{sec:proof1} we prove that the maximizer of the analytic continuation
error can be obtained from the solution of an integral equation. In
section~\ref{extrapolation from boundary SECT} we analyze the case when
$\GG=[-1,1]$ lies on the boundary of $\bb{H}_{+}$. In this case we show that
the error maximizer also solves an integral equation, but with a singular,
non-compact integral operator. This singular equation is then solved explicitly
and the exponent $\Gg(z)$ is computed. Examining the formula for $\Gg(z)$ we
find a beautiful geometric interpretation of this exponent.


\section{Main Results}
\setcounter{equation}{0} 
\label{sec:main}
\noindent \textbf{Notation:} Let us write $A \sim B$ as $\epsilon \to 0$,
whenever $\displaystyle \lim_{\epsilon \to 0} A/B = 1$. Let us also write $A
\lesssim B$, if there exists a constant $c$ such that $A \leq c B$ and
likewise the notation $A \gtrsim B$ will be used. If both $A
\lesssim B$ and $A \gtrsim B$ are satisfied we will write $A\simeq B$.

Let $\Gamma \subset \HH_+$ be a compact smooth ($C^{1}$) curve.
Let $L^2(\Gamma) := L^2(\Gamma, |\D \zeta|)$. In this paper all the $L^2$ spaces will be spaces of complex valued functions. Consider the Hardy space 

\begin{equation}\label{H2def}
H^2=H^2(\HH_+) := \{f \ \text{is analytic in} \ \HH_+: \sup_{y>0} \|f(\cdot + iy)\|_{L^2(\RR)} < \infty \}.
\end{equation}

\noindent It is well known \cite{koos98} that a function $f \in H^2$ has $L^2$ boundary data and that $\|f\|_{H^2} = \|f\|_{L^2(\RR)}$ defines a norm in $H^2$.

\vspace{.1in}

\subsection{Interior}

\begin{theorem}[Interior] \label{upper bound THM}
Let $\Gamma \Subset \HH_+$ be a smooth ($C^1$), bounded and simple curve and
$z \in \HH_+ \setminus\Gamma$ be the extrapolation point. Let $\epsilon > 0$ and $f \in H^2$ be such that $\|f\|_{H^2} \leq 1$ and $\|f\|_{L^2(\Gamma)} \leq \epsilon$, then

\begin{equation} \label{main bound}
  |f(z)| \leq \frac{3}{2}u_{\epsilon, z}(z)\min\left\{\frac{1}{\|u_{\epsilon, z}\|_{H^2}},
\frac{\Ge}{\|u_{\epsilon, z}\|_{L^{2}(\GG)}}\right\},
\end{equation}

\noindent where $u_{\epsilon, z}$ solves the integral equation

\begin{equation} \label{Ku +eps^2 u = kz}
(\CK u)(\Gz) +\epsilon^2 u(\Gz) = p_z(\Gz), \qquad\Gz\in \Gamma,
\end{equation}

\noindent with

\begin{equation} \label{K and kz}
(\CK u) (\zeta) = \frac{1}{2\pi} \int_\Gamma \frac{i u(\tau)}{\zeta - \overline{\tau}} |\D \tau|, \qquad \qquad p_z(\zeta) = \frac{i}{\zeta- \overline{z}}
\end{equation}

\end{theorem}
The theorem is proved in Section~\ref{DMP Section}.

\begin{remark} \label{remark} \mbox{}
\begin{enumerate}

\item $\CK$ is a compact, self-adjoint and positive operator on $L^2(\Gamma)$ (cf. Section~\ref{conjecture implies optimality SECTION}). In particular, \eqref{Ku +eps^2 u = kz} has a unique solution $u_{\Ge,z} \in L^2(\Gamma)$.

\item It is evident that
  $\CK u$ and $ p_z$ are well-defined members of $H^{2}(\bb{H}_{+})$. Hence,
  when $\Gz\not\in\GG$ the integral equation (\ref{Ku +eps^2 u = kz}) is
  interpreted as a definition of $u_{\epsilon, z} (\Gz)$. 
  This explains the meaning of
  the \rhs\ in \eqref{main bound}.
\end{enumerate}
\end{remark}

The bound in (\ref{main bound}) is asymptotically optimal, as $\epsilon\to 0$
since the function 
\begin{equation}
  \label{Mepsz}
  M_{\epsilon, z} (\zeta)= u_{\epsilon, z}(\zeta)\min\left\{\frac{1}{\|u_{\epsilon, z}\|_{H^2}},
\frac{\Ge}{\|u_{\epsilon, z}\|_{L^{2}(\GG)}}\right\}
\end{equation}
has $L^2(\Gamma)$-norm bounded by $\epsilon$ and $H^{2}$-norm bounded by 1. 
Even though we only required $f$ to be in $H^2(\HH_+)$, the optimal function
  $M_{\epsilon, z}(\Gz)$ is actually analytic in $\bb{C}\setminus\bra{\GG}$. 

  We believe that the two quantities under the minimum in (\ref{Mepsz}) have
  the same asymptotics as $\Ge\to 0$, and hence, the error maximizer can be
  written either as $u_{\epsilon, z}/\|u_{\epsilon, z}\|_{H^2}$ or as $\Ge
  u_{\epsilon, z}/\|u_{\epsilon, z}\|_{L^2(\GG)}$.

\begin{conjecture} \label{conjecture upper bound}

Let $u_{\epsilon, z}$ be as in Theorem~\ref{upper bound THM}, then 

\begin{equation} \label{optimality eqn}
\|u_{\epsilon, z}\|_{L^2(\Gamma)} \simeq \epsilon \|u_{\epsilon, z}\|_{H^2}.
\end{equation}
\end{conjecture}
The difficulty in establishing (\ref{optimality eqn}) is that in this
particular equation the solution $u_{\epsilon, z}$ must achieve a delicate
balance after the cancellation in (\ref{Ku +eps^2 u = kz}). We will show
(see Section~\ref{SEC K norm estimates}) that
$\|u_{\Ge,z}\|_{L^{2}(\GG)}=o(\Ge^{-1})$, while
\[
\lim_{\Ge\to 0}\CK u_{\Ge,z}=p_{z}
\]
both in $L^{2}(\GG)$ and pointwise in $\bb{H}_{+}$. Therefore, the second term on the \lhs\ in
(\ref{Ku +eps^2 u = kz}) is infinitesimal compared to other terms and hence
represents a delicate matching of the remainder after cancellation in
$p_{z}-\CK u_{\Ge,z}=\Ge^{2}u_{\Ge,z}=o(\Ge)$ in $L^{2}(\GG)$. We will also
see that $M_{\epsilon, z}(z)=o(1)$ and $M_{\epsilon, z}(z)\gg\Ge$ as $\epsilon \downarrow 0$.
This implies that if the power law transition principle
holds, i.e. 
\begin{equation}
\label{powerlaw}
 M_{\epsilon, z} (z) \sim  \epsilon^{\gamma}, \qquad \text{as} \qquad \epsilon \to 0,
\end{equation}
then $\gamma = \gamma_{\GG}(z) \in (0,1)$. In (\ref{powerlaw}) we
abuse our notation convention for $\sim$ for the sake of aesthetics. The
mathematically correct statement would be $\ln M_{\epsilon, z} (z)
\sim\Gg_{\GG}(z)\ln\Ge$. The exponent $\gamma_{\GG}(z)$ is expected to grow
smaller the further away point $z$ moves from $\GG$, so that $\Gg_{\GG}(z)\to
0$ as $z\to\infty$. 
The genesis of the exponent $\Gg_{\GG}(z)$ in (\ref{powerlaw}) from equation
\eqref{Ku +eps^2 u = kz} that itself contains no fractional exponents of
$\Ge$, comes from the conjectured exponential decay of eigenvalues $\Gl_{n}$
of $\CK$. 

The exponential upper bound on $\Gl_{n}$ is a consequence of the displacement
rank 1 structure:
\begin{equation}
  \label{displ}
  (M \CK -\CK M^{*})u=\frac{i}{2\pi}\int_{\GG}u(\tau)|d\tau|=:Ru,
\end{equation}
where $M:L^{2}(\GG)\to L^{2}(\GG)$ is the operator of multiplication by
$\tau\in\GG$: $(Mu)(\tau)=\tau u(\tau)$. 
The operator $R$ on the \rhs\ of (\ref{displ}) is a rank-one operator, since its range
consists of constant functions. Then, according to \cite{beto17},
\begin{equation}
  \label{Zolot}
  \Gl_{n+1}\le
  \rho_{1}\Gl_{n},\qquad\rho_{1}=\inf_{r\in\CM}\frac{\max_{\tau\in\GG}|r(\tau)|}
{\min_{\tau\in\GG}|r(\bra{\tau})|},
\end{equation}
for all $n\ge 1$, where $\CM$ is the set of all M\"obius transformations
\[
r(\tau)=\frac{a\tau+b}{c\tau+d}.
\]
It is easy to see that $\rho_{1}<1$ by considering M\"obius transformations that
map upper half-plane into the unit disk. Then $\GG$ will be mapped to a
curve inside the unit disk, so that $m=\max_{\tau\in\GG}|r(\tau)|<1$. By the
symmetry property of M\"obius transformations the image of $\bra{\GG}$ will be
symmetric to the image of $\GG$ with respect to the inversion in the unit
circle. Thus, $\min_{\tau\in\GG}|r(\bra{\tau})|=1/m$, so that $\rho_{1}\le m^{2}<1$.
In particular this implies that all eigenvalues have
multiplicity 1. 

The implied exponential upper bound $\Gl_{n+1}\le \rho_{1}^{n}\Gl_{1}$ is not the best
that one can derive from the rank-1 displacement structure
(\ref{displ}). According to a theorem of Beckermann and Townsend \cite{beto17},
$
\Gl_{n}\le Z_{n}(\GG,\bra{\GG})\Gl_{1},
$
where $Z_{n}$ is the $n$th Zolotarev
number \cite{zolo1877}. When $n$ is large, the Zolotarev numbers decay exponentially
$\ln Z_{n}(\GG,\bra{\GG})\sim -n\ln \rho_{\GG}$,
where $\rho_{\GG}$ is the Riemann invariant, whereby the annulus
$\{1<|z|<\rho_{\GG}\}$ is conformally equivalent to the Riemann sphere with
$\GG$ and $\bra{\GG}$ removed \cite{gonc69}. Hence,
\begin{equation}
  \label{BTub}
  \Gl_{n}\lesssim\rho_{\GG}^{-n}.
\end{equation}

We are ready now to relate the spectral exponential decay rates to the power
law (\ref{powerlaw}).  Let $\{e_n\}_{n=1}^\infty$ denote the orthonormal
eigenbasis of $\CK$. In this basis equation \eqref{Ku +eps^2 u = kz}
diagonalizes:
\[
\Gl_{n}u_{n}+\Ge^{2}u_{n}=\pi_{n},\qquad u_{n}=(u,e_{n})_{L^{2}(\GG)},\quad
\pi_n = (p_z, e_n)_{L^2(\Gamma)},
\]
and is easily solved
\begin{equation}
  \label{unsol}
  u_{n}=\frac{\pi_{n}}{\Gl_{n}+\Ge^{2}}.
\end{equation}
We will prove that 
\begin{equation}
  \label{pi_n_evidence}
\sum_{n=1}^{\infty}\frac{|\pi_{n}|^{2}}{\Gl_{n}}<+\infty, \qquad
\sum_{n=1}^{\infty}\frac{|\pi_{n}|^{2}}{\Gl^{2}_{n}}=+\infty,
\end{equation}
indicating that the coefficients $\pi_{n}$ must also decay exponentially fast. The
power law principle is then a consequence of the strictly exponential decay of
eigenvalues $\Gl_{n}$ and coefficients $\pi_{n}$.

\begin{theorem} \label{connections THM}
Let $\{e_n\}_{n=1}^\infty$ denote the orthonormal eigenbasis of $\CK$, and let 
$\pi_n = (p_z, e_n)_{L^2(\Gamma)}$. Assume that
\begin{equation} \label{lambda_n and pi_n asymp}
\lambda_n \sim C_{1}e^{-\alpha n}, \quad |\pi_n|^2 \sim C_{2}e^{-\beta n},\qquad
0<\Ga<\Gb<2\Ga,
\end{equation}
so that \eqref{pi_n_evidence} holds.
Then estimate \eqref{optimality eqn} holds, and $M_{\epsilon, z}$, given by
\eqref{Mepsz} has the power law asymptotics
\begin{equation}
  \label{exactgamma}
  M_{\epsilon, z}(z) \simeq \epsilon^{\frac{\beta - \alpha}{\alpha}},
\end{equation}
with implicit constants independent of $\epsilon$.
\end{theorem}
The theorem is proved in Section~\ref{SEC K from exp decay to power law}.
\begin{remark}
  The coefficients $\pi_{n}$ of $p_{z}$ in the eigenbasis of $\CK$
  can be expressed in terms of the eigenfunctions $\{e_{n}\}$ (cf. \eqref{(u,p)=Ku}):
\[
\pi_n = 2\pi \lambda_n \overline{e_n(z)}.
\]   
\end{remark}

\begin{conjecture} \label{conjecture M_eps,z asymp} The eigenvalues $\Gl_{n}$
  of $\CK$ and coefficients $\pi_n = 2\pi \lambda_n \overline{e_n(z)}$
  have exponential decay asymptotics \eqref{lambda_n and pi_n
    asymp}. Moreover, we also conjecture that the asymptotic upper bound (\ref{BTub})
  captures the rate of exponential decay of $\Gl_{n}$, i.e. $\Ga=\ln\rho_{\GG}$.
\end{conjecture}
There is substantial evidence supporting this conjecture, including the
explicit formula for $\Gg(z)$ in the limiting case when
$\GG\subset\Md\bb{H}_{+}$, given in Theorem~\ref{main boundary THM}
below. Also, if the $L^{2}$ norm of $f\in H^{2}$ were of order $\Ge$ on a
compact subdomain $G\subset\bb{H}_{+}$, instead of the curve $\GG$, then the
conjectured asymptotics of $\Gl_{n}$ would hold, as shown in \cite{parfenov},
provided the boundary of $G$ is sufficiently smooth. The curve $\GG$ could
also be regarded as a limiting case of a domain. However, its boundary would
not be smooth and the analysis in \cite{parfenov} would not apply.

The operator $\Ge^{2}+\CK$ in the integral equation (\ref{Ku +eps^2 u = kz})
is almost singular when $\Ge$ is small, since $\CK$ is compact and has no
bounded inverse. It was the idea of Leslie Greengard to solve (\ref{Ku +eps^2
  u = kz}) directly numerically using quadruple precision floating point
arithmetic available in FORTRAN. He has written the code and shared the
FORTRAN libraries for Gauss quadrature, linear systems solver and eigenvalues
and eigenvectors routines for Hermitian matrices. For the numerical
computations we took $\GG=[-1,1]+ih$, and extrapolation points
$z+ih$, $z\ge 1$. Quadruple precision allowed us to compute all eigenvalues of
$\CK$ that are larger than $10^{-33}$ and solve the integral equation (\ref{Ku
  +eps^2 u = kz}) for values of $\Ge$ as low as $10^{-16}$. For this
particular choice of $\GG$ the operator $\CK$ is a finite convolution type
operator with kernel $k(t) = \frac{i(2\pi)^{-1}}{t+2ih}$. Asymptotics of
eigenvalues of positive self-adjoint finite convolution operators with
real-valued kernels (i.e. even real functions $k(t)$) were obtained by Widom
in \cite{widom64}. To apply these results we note that $\widehat{k}(\xi) =
e^{-2h\xi} \chi_{(0,+\infty)}(\xi)$, which has exact exponential decay when
$\xi\to+\infty$. The operator $\CK_{0}$ with the even real kernel
$k_{0}(t)=2\Re k(t)$ has symbol $\widehat{k}_{0}(\xi) = e^{-2h|\xi|}$ to which
Widom's theory applies.  Widom's formula gives
\begin{equation*}
\ln \lambda_n(\CK_0) \sim -W n,\qquad\text{as} \ n\to \infty,\qquad 
W=-\pi \frac{K\left(\sech\left(\frac{\pi}{2h}\right)\right)}
{K\left(\tanh\left(\frac{\pi}{2h}\right)\right)},
\end{equation*}
\noindent where $K(k)$ is the complete elliptic integral of the first kind. We therefore
obtain an upper bound
\begin{equation} \label{Widom upper bound}
\ln \lambda_n(\CK_h) \leq \ln \lambda_n(\CK_0) \sim -Wn.
\end{equation}
The lower bound can be obtained from the same formula using an inequality
\[
\lambda_n(\CK_0) \leq \lambda_{n/2}(\CK_h) +
\lambda_{n/2}(\overline{\CK_h}) = 2 \lambda_{n/2}(\CK_h),
\] 
so that
\begin{equation} \label{Widom lower bound}
\ln \lambda_n(\CK_h) \geq \ln \tfrac{1}{2} + \ln \lambda_{2n}(\CK_0) \sim  -2Wn.
\end{equation}
Figure~\ref{fig:Keigs}, where $h=1$ supports the exponential decay conjecture
\eqref{lambda_n and pi_n asymp} and shows that estimates (\ref{Widom upper
  bound}), (\ref{Widom lower bound}) are not asymptotically sharp. By
contrast, Figure~\ref{fig:Keigs} shows that the Beckermann-Townsend upper
bound (\ref{BTub}) matches the asymptotics of $\Gl_{n}$ very well. The
explicit transformation $\Psi$ of the extended complex plane with $[-1,1]\pm
ih$ removed onto the annulus
$\{v\in\bb{C}:\rho_{\GG}^{-1/2}<v<\rho_{\GG}^{1/2}\}$ has been derived in
\cite[p.~138]{akhi90} in terms of the elliptic functions and integrals
\begin{equation}
  \label{Psi}
  \Psi^{-1}(v)=\frac{h}{\pi}\left(\Gz\left(\left.\frac{\ln v}{2\pi i}\right|\tau\right)
-\Gz\left(\left.\hf\right|\tau\right)\frac{\ln v}{\pi i}\right),
\qquad\tau=\frac{K(1-m)}{K(m)},
\end{equation}
where $\Gz(z|\tau)$ is the Weierstrass zeta function with quasi-periods $1$ and
$i\tau$. The Riemann invariant $\rho_{\GG}=e^{2\pi\tau}$ is computed after
finding the unique solution $m\in(0,1)$ of 
\[
K(m)E(x(m)|m)-E(m)F(x(m)|m)=\frac{\pi}{2h},\quad
x(m)=\sqrt{\frac{K(m)-E(m)}{mK(m)}}.
\]
\begin{figure}[t]
  \begin{subfigure}[t]{1.7in}
\includegraphics[scale=0.125]{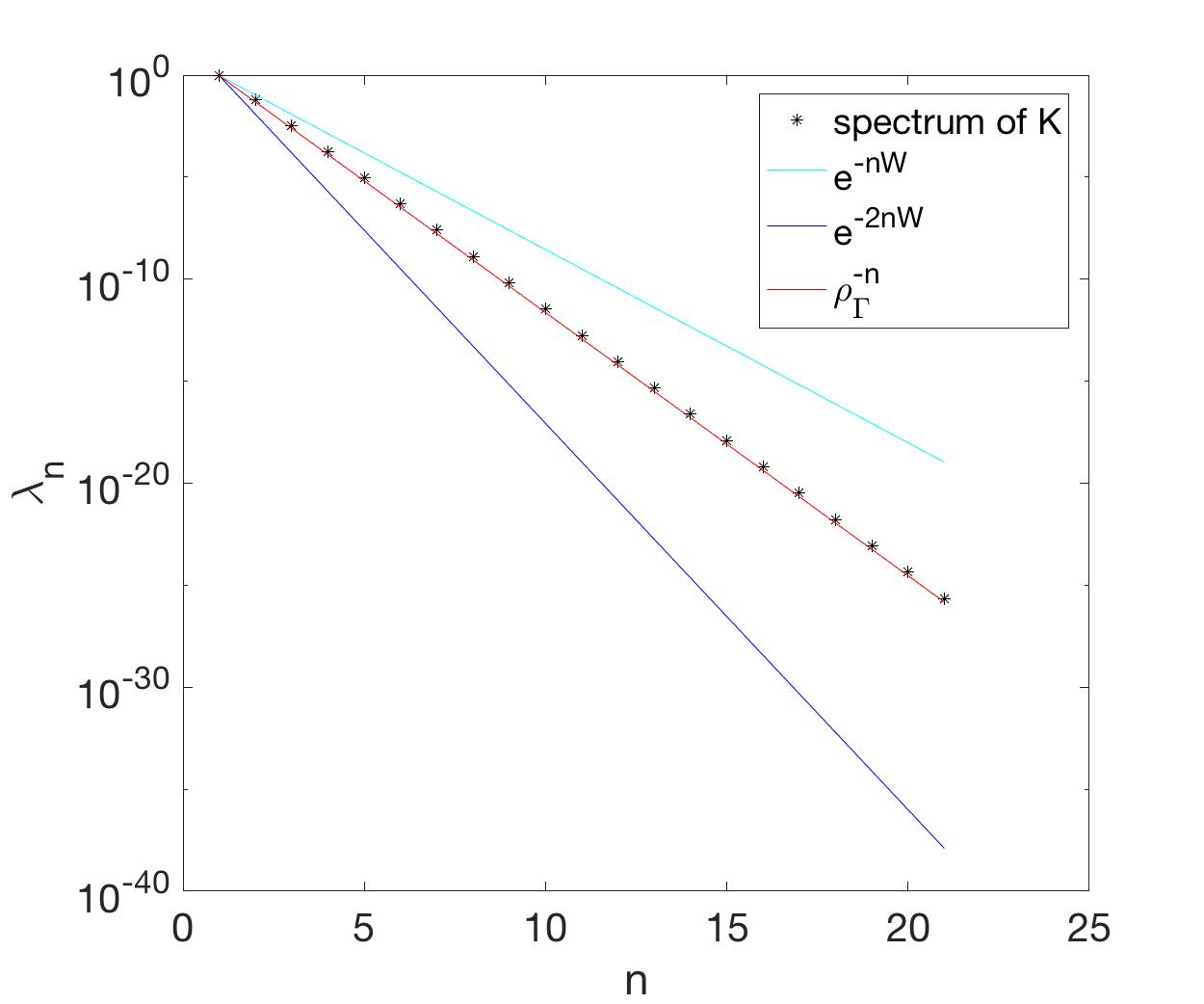}
\caption{}
\label{fig:Keigs}
  \end{subfigure}\hspace{4ex}
  \begin{subfigure}[t]{1.7in}
\includegraphics[scale=0.15]{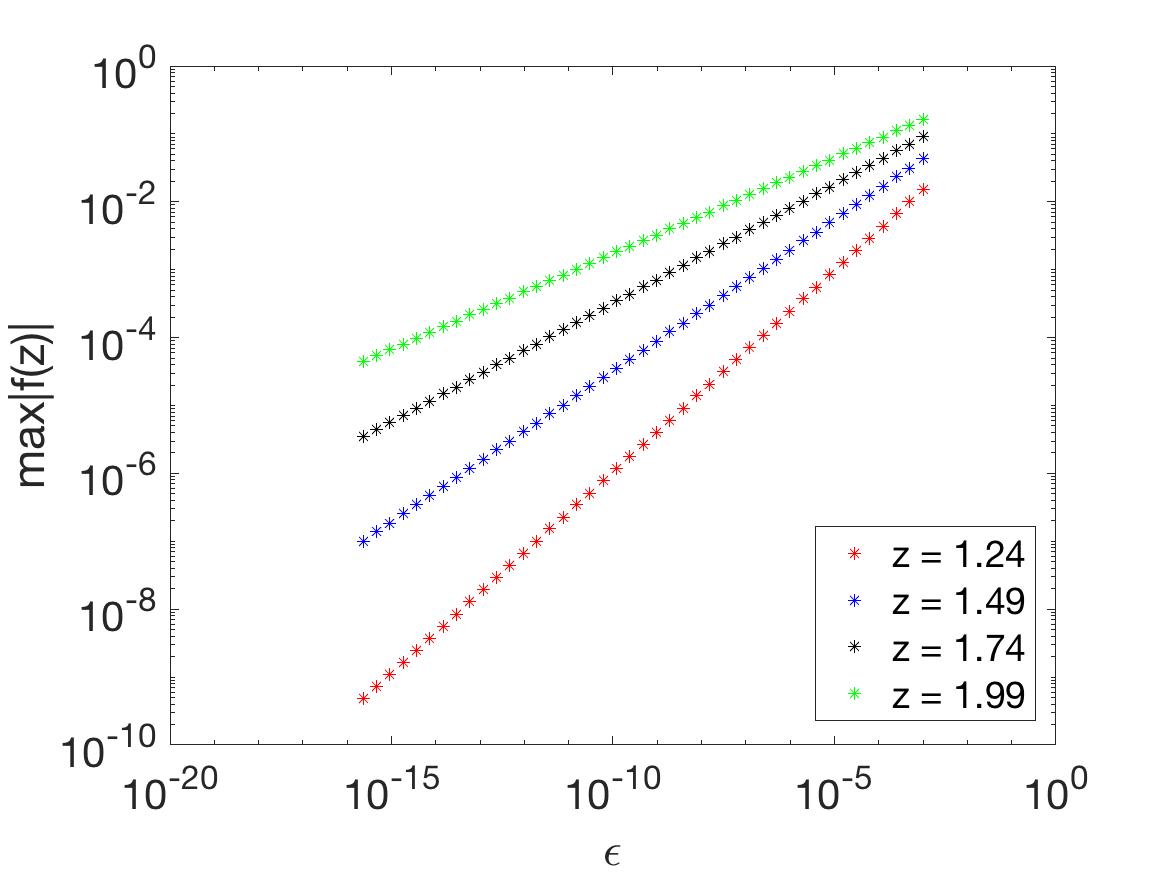}
\caption{}
\label{fig:powerlaw}
  \end{subfigure}\hspace{7ex}
\begin{subfigure}[t]{1.7in}
\includegraphics[scale=0.15]{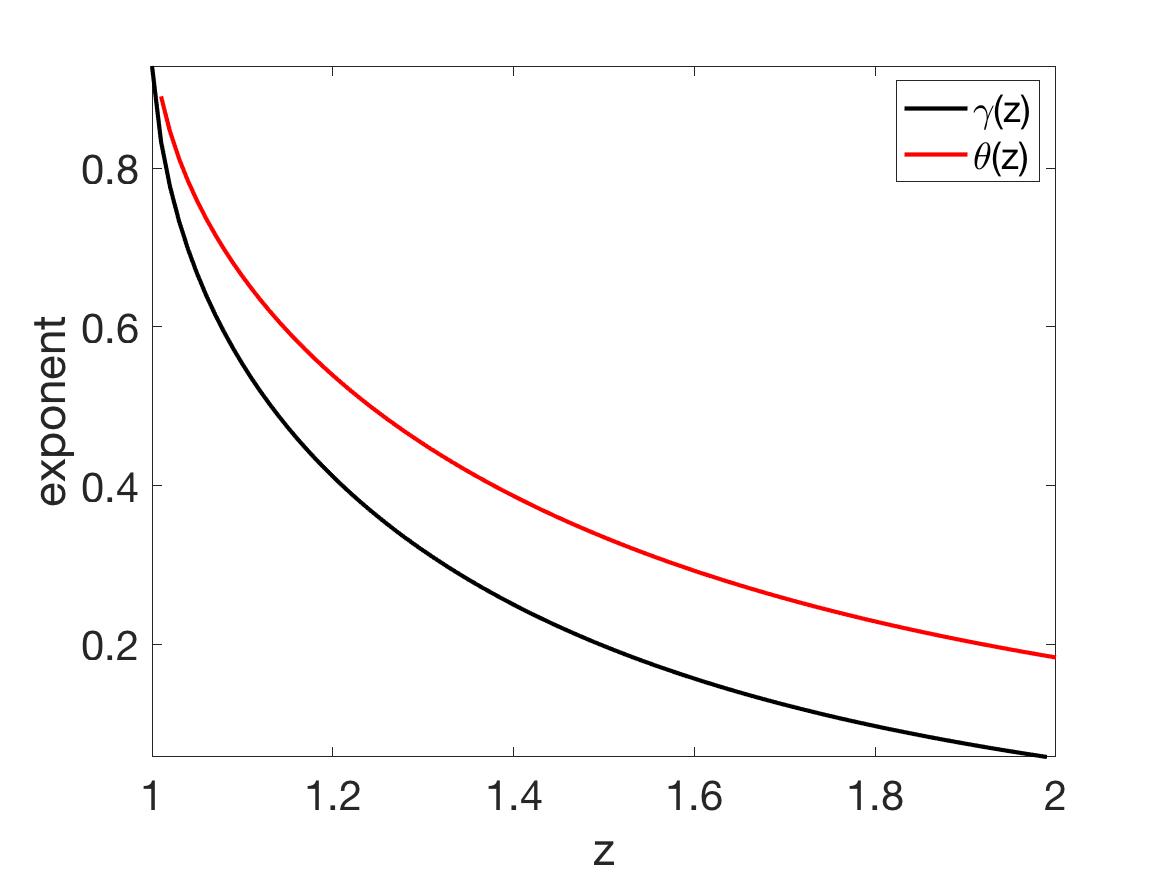}
\caption{}
\label{fig:aofz}
  \end{subfigure}
  \caption{Numerical support for the power law transition principle.}
\end{figure}

We can show by a specific construction that one cannot expect better precision
at a point $z$ than $\Ge^{\Gth(z)}$ for some $\Gth(z)\in(0,1)$, giving an
upper bound on $\Gg(z)$. This is done by mapping the explicit eigenfunction
expansion of the solution of the integral equation for the annulus problem to
the upper half-plane by the explicit conformal transformation $\Psi$ (see
\cite{grho-annulus,grho-CEMP} for details). This gives the estimate
\begin{equation} \label{gamma1}
\Gth(z) = \frac{\ln |\Psi(\bra{z})|}{\pi\tau} \in (0,1),
\end{equation}
achieved by the function
\begin{equation} \label{Uguess}
f(\Gz) = \frac{\epsilon^{2-\Gth(z)}}{\Gz+ih} \sum_{n=1}^{\infty}
\frac{\left(\bra{\Psi(z)}\Psi(\Gz)\right)^{n}}{\Ge^{2} + \rho^{-n}}\in H^{2}(\bb{H}_{+}).
\end{equation}
Figure~\ref{fig:powerlaw} shows values of $M_{\Ge,z}(z)$ as a function of
$\Ge$, supporting the power law principle \eqref{powerlaw}. We also compare
the computed exponents $\Gg_{\GG}(z)$ with the estimate (\ref{gamma1}) for
$\GG=[-1,1]+0.5i$, and extrapolation points $z+0.5i$, $z>1$.
Figure~\ref{fig:aofz} shows $\Gg(z)$ (obtained by least squares linear fit of
the data for various values of $z$, four of which are shown in
Figure~\ref{fig:powerlaw}) and the upper bound $\Gth(z)$ given by
(\ref{gamma1}). We remark that by virtue of transplanting the actual maximizer
of $|f(z)|$ from one geometry to the other, the structure of the test function
(\ref{Uguess}) resembles the optimal one (\ref{unsol}). In fact, for values of
$h>0.6$ the bound $\Gth(z)$ is virtually indistinguishable from $\Gg_{\GG}(z)$.


\subsection{Boundary} \label{SEC K boundary results}

\noindent We recall that functions in the Hardy space $H^{2}$ (see
\eqref{H2def}) are determined uniquely not only by their values on any curve
$\Gamma \subset \HH_+$, but also on $\Gamma \subset \RR$. Indeed, if $f = 0$ on $\Gamma
\subset \RR$, the Cauchy integral representation formula implies
\[
f(z) = \frac{1}{2\pi i} \int_{\Gamma^c} \frac{f(t) \D t}{t-z},\qquad
z \in \HH_+,
\] 
where $\Gamma^c=\bb{R}\setminus\Gamma$. Then $f(z)$ has analytic extension to
$\bb{C} \setminus \Gamma^c$, which vanishes on a curve $\Gamma$ inside its
domain of analyticity and therefore $f \equiv 0$. This rigidity property
suggests that we should expect the same
power law behavior of the analytic continuation error as for the curves in the
interior of $\HH_+$. 

We will consider the most basic case when $\Gamma \subset \RR$ is an interval.
(By rescaling  and translation we may assume, \WLOG, that $\Gamma = [-1,1]$).
We proceed by representing $\GG$ as a limit of interior curves $\Gamma_h =
[-1,1] + ih$ as $h \downarrow 0$. For curves $\Gamma_h$, Theorem~\ref{upper
  bound THM} can be applied and in the resulting upper bound and the integral
equation, limits, as $h \downarrow 0$, can be taken. As a result we obtain
\begin{theorem}[Boundary] \label{main boundary THM}
Let $z=z_r + i z_i \in \HH_+$ and $\epsilon \in (0,1)$. Assume $f \in H^2$ is such that $\|f\|_{H^2} \leq 1$ and $\|f\|_{L^2(-1,1)} \leq \epsilon$, then

\begin{equation} \label{main bound boundary explicit}
|f(z)| \leq \rho \epsilon^{\gamma(z)}
\end{equation}

\noindent where $\rho^{-2} = \frac{z_i}{9} \left(\arctan \frac{z_r + 1}{z_i} - \arctan \frac{z_r - 1}{z_i}\right)$ and

\begin{equation} \label{gamma(z)}
\gamma(z) = -\frac{1}{\pi} \arg \frac{z+1}{z-1} \in (0,1)
\end{equation}

\noindent is the angular size of $[-1,1]$ as seen from $z$, measured in units
of $\pi$ radians. Moreover, the upper bound \eqref{main bound boundary
  explicit} is asymptotically (in $\epsilon$) optimal and the maximizer that
attains the bound up to a multiplicative constant independent of $\epsilon$ is
\begin{equation} \label{maximizer boundary}
W(\zeta) = \epsilon \frac{p(\zeta)}{\|p\|_{L^2(-1,1)}} e^{\frac{i }{\pi} \ln \epsilon \ln \frac{1+\zeta}{1-\zeta}}, \qquad \qquad \zeta \in \HH_+
\end{equation} 

\noindent where $p(\zeta) = i / (\zeta - \overline{z})$ and $\ln$ denotes the principal branch of logarithm.

\end{theorem}
The theorem is proved in Section~\ref{extrapolation from boundary SECT}.

\vspace{.1in}

\begin{remark}\mbox{}
\begin{enumerate} 
\item Our explicit formulas show that the problem of predicting the value of a
  function at $z = z_0 \in \RR \setminus [-1,1]$ is ill-posed in every
  sense. Indeed, in the optimal bound \eqref{main bound boundary explicit} $\rho \to +\infty$ and $\gamma(z) \to 0$ as $z \to z_0$. 

\item The set of points $z\in\bb{H}_{+}$ for which $\Gg(z)$ is the same is an arc
of a circle passing through $-1$ and $1$ that lies in the upper half-plane. 
\end{enumerate}
\end{remark}


\section{A priori estimates} \label{SEC K norm estimates}
\setcounter{equation}{0}

\noindent \textbf{Notation:} In this section let $\|\cdot\|$ and $(\cdot, \cdot)$ be the norm and the inner product of $L^2(\Gamma)$.

\vspace{.1in}

In this section we prove several general properties of solutions of the
integral equation (\ref{Ku +eps^2 u = kz}). They show that the solution cannot
depend analytically on $\Ge$, as $\Ge\to 0$.

\begin{lemma} \label{LEM L^2 and H^2 norms of u}
Let $u_\epsilon$ solve $\CK u + \epsilon^2 u = p_z$. Then as $\epsilon \downarrow 0$
\begin{itemize}
\item[(i)] $\displaystyle \|u_\epsilon\| \to \infty$
\item[(ii)] $\dfrac{u_{\Ge}}{\|u_{\Ge}\|}\weak 0$ in $L^{2}(\GG)$
\item[(iii)] $\dfrac{u_{\Ge}}{\|u_{\Ge}\|_{H^{2}}}\weak 0$ in $L^{2}(\bb{R})$
\end{itemize}
\end{lemma}
\begin{proof}
Part (i). Recalling formula (\ref{unsol}), we have
\begin{equation}
  \label{L2normu}
  \|u_\epsilon\|^{2}=\sum_{n=1}^\infty|u_{n}|^{2}=\sum_{n=1}^\infty \frac{|\pi_n|^2}{(\Ge^{2}+\lambda_{n})^2},
\end{equation}
and applying Lemma Fatou we conclude that
\[
\sum_{n=1}^\infty\frac{|\pi_n|^2}{\lambda_{n}^2}\le\limi_{\Ge\to
  0}\|u_{\Ge}\|^{2} .
\]
Hence, boundedness along any subsequence of $\|u_\epsilon\|$ implies
convergence of the series above, which leads to a contradiction, as observed
in Lemma~\ref{upper bound sums THM}.

Part (ii). Let $v_{\Ge}=u_{\Ge}/\|u_{\Ge}\|$. Extracting a weakly convergent
subsequence in $L^2(\Gamma)$ and passing to the limit in
\[
\CK v_{\Ge}+\Ge^{2}v_{\Ge}=\frac{p_{z}}{\|u_{\Ge}\|},
\]
using Part (i) of the lemma, we obtain the equation for the weak limit
$v_{0}$: $\CK v_{0}=0$. Hence, by Lemma~\ref{lem:desnisty},  $v_{0}=0$.
Since every weakly convergent subsequence of $v_{\Ge}$ has a zero limit, the
entire family $v_{\Ge}$ converges weakly to 0. 

Part (iii). Let now $v_\epsilon=u_{\Ge}/\|u_{\Ge}\|_{H^{2}}$, and let
$v_{\Ge_{k}}\weak v_{0}$ in $H^{2}(\bb{H}_{+})$. Then passing to the limit in
(\ref{H2rep}) we obtain
\[
v_{0}(\Gz)=\nth{2\pi i}\int_{\bb{R}}\frac{v_{0}(x)dx}{x-\Gz}.
\]
Observing that $\|u_{\Ge}\|_{H^{2}} \geq c \|u_{\Ge}\| \to \infty$, and
repeating the argument in the proof of Part (ii) of the lemma, we get
$\CK v_0 = 0$ on $\Gamma$. Hence $v_0 = 0$ on $\Gamma$, and by analyticity, $v_0 =
0$ on $\RR$.
\end{proof}

Lemma~\ref{LEM L^2 and H^2 norms of u} has a number of immediate corollaries,
especially when combined with formula (\ref{Pyth}) (see
Corollary~\ref{cor:Pyth}) and the Cauchy representation formula for $H^{2}$
functions (\ref{H2rep}).
Using the Cauchy representation formula (\ref{H2rep}) part~(iii) implies that
$u_{\Ge}(\Gz)/\|u_{\Ge}\|_{H^{2}}\to 0$, as $\Ge\to 0$ for all
$\Gz\in\bb{H}_{+}$. In particular, $u_{\Ge}(z)/\|u_{\Ge}\|_{H^{2}}\to 0$, as
$\Ge\to 0$. Applying this fact to (\ref{Pyth}) we conclude that
$\|u_{\Ge}\|_{H^{2}}=o(\Ge^{-2})$ and that
\begin{equation}
  \label{oeps}
  \|u_{\Ge}\|^{2}=o(\|u_{\Ge}\|_{H^{2}})=o(\Ge^{-2}),
\end{equation}
showing that $\|u_{\Ge}\|=o(\Ge^{-1})$ and hence
$u_{\Ge}(z)=o(\Ge^{-2})$. From the integral
equation (\ref{Ku +eps^2 u = kz}) we obtain
\begin{equation}
  \label{canceff}
  p_{z}-\CK u_{\Ge,z}=\Ge^{2}u_{\Ge,z}=o(\Ge)
\end{equation}
in $L^{2}(\GG)$.
Returning to the Cauchy representation formula we
conclude from (\ref{oeps}): $\Ge^{2}u_{\Ge}(\Gz)\to
0$, as $\Ge\to 0$ for all $\Gz\in\bb{H}_{+}$. Hence, $p_{z}-\CK u_{\Ge,z}\to 0$
pointwise in $\bb{H}_{+}$.



\section{Justification of the power law transition principle} \label{conjecture implies optimality SECTION}
\setcounter{equation}{0}
In this section we prove Theorem~\ref{connections THM} under slightly more
general assumptions. It shows how exponential decay of eigenvalues and
eigenfunctions gives rise to power law estimates (\ref{powerlaw}).
Throughout this section $\|\cdot\|$ and $(\cdot, \cdot)$ denote the norm and the inner product of the space $L^2(\Gamma)$.

\vspace{.1in}

\subsection{Spectral representation of $u_{\Ge,z}(z)$}
We begin by rewriting the value $u_{\Ge,z}(z)$, in terms of $\lambda_n$ and
coefficients $\pi_n$ of $p_z$ in the eigenbasis $\{e_n\}$.
\begin{lemma} \label{upper bound sums THM}
Let $u_{\Ge,z}$ be the solution of \eqref{Ku +eps^2 u = kz}. Then
\begin{equation}
  \label{uofz}
2\pi u_{\epsilon, z}(z)=\sum_{n=1}^\infty\frac{|\pi_n|^2}{\lambda_n(\lambda_n+\epsilon^2)}.
\end{equation}
\end{lemma}
\begin{proof}
   Observe that 
   
\begin{equation} \label{(u,p)=Ku}
(u,p_{z})=\int_{\Gamma} u(\tau)\frac{i}{z-\overline{\tau}}|d\tau|=2\pi(\CK u)(z)
\end{equation}

\noindent for any $u \in L^2(\Gamma)$, therefore for the solution $u_{\epsilon, z}$ of  \eqref{Ku +eps^2 u = kz} we have
\begin{equation} \label{u(z) from int eq}
2\pi u_{\epsilon, z}(z) = \frac{1}{\epsilon^2}(2\pi p_{z}(z)-2\pi(\CK u_{\epsilon, z})(z))=
\frac{1}{\epsilon^2}\left( \frac{\pi}{\Im z} - (u_{\epsilon, z}, p_z) \right).
\end{equation}
Since
\[
(u_{\epsilon, z}, p_z)=\sum_{n=1}^\infty\frac{|\pi_n|^2}{\lambda_n+\epsilon^2},
\]
it is easy to see that \eqref{uofz} is equivalent to
\begin{equation} \label{(K^-1 p_z, p_z) and pi_j}
\frac{\pi}{\Im z} =  \sum_{n=1}^\infty \frac{|\pi_n|^2}{\lambda_n}.
\end{equation}
Formally the series on the \rhs\ of \eqref{(K^-1 p_z, p_z) and pi_j} can be
written as
\[
\sum_{n=1}^\infty \frac{|\pi_n|^2}{\lambda_n}=(\CK^{-1}p_{z},p_{z}).
\]
However, it is easy to see that $p_{z}$ is not in the range of $\CK$. Indeed,
for any $u\in L^{2}(\GG)$ its image $f(\Gz)=(\CK u)(\Gz)$ has an analytic
extension to $\bb{C}\setminus\bra{\GG}$, while $p_{z}(\Gz)$ has a pole at
$\bra{z}\not\in\bra{\GG}$. As a consequence
\begin{equation}
  \label{diverge}
  \sum_{n=1}^{\infty}\frac{|\pi_{n}|^{2}}{\Gl^{2}_{n}}=+\infty,
\end{equation}
since otherwise the function
\[
v=\sum_{n=1}^{\infty}\frac{\pi_{n}}{\Gl_{n}}e_{n}
\]
would belong to $L^{2}(\GG)$ and would have the property $\CK v=p_{z}$.

The key to understanding the operator $\CK$
is the observation that its range 
\[
\CR(\CK)=\{f:f=\CK u,\ u\in L^{2}(\GG)\}
\]
consists of functions that have an analytic extension to functions in
$H^{2}(\bb{H}_{+})$, moreover for any $f\in H^{2}(\bb{H}_{+})$ and
$u\in L^{2}(\GG)$ we have
\begin{equation}
  \label{keyK}
  (u,f)=(\CK u,f)_{H^{2}}.
\end{equation}
Indeed, changing the order of integration we obtain
\begin{equation*}
(\CK u,f)_{H^{2}}=\int_{\RR} (\CK u)(x)\bra{f(x)}\D x = \int_{\Gamma}u(\tau) \bra{
\left(\nth{2\pi i}\int_\RR \frac{f(x)\D x}{x-\tau}\right)} |\D \tau|=
\int_{\Gamma} u(\tau) \overline{f(\tau)} |\D \tau| =(u, f),
\end{equation*}
where we have used the Cauchy representation formula for $H^{2}(\bb{H}_{+})$
functions in terms of their boundary values:
\begin{equation}
  \label{H2rep}
  f(\Gz)=\nth{2\pi i}\int_\RR \frac{f(x)\D x}{x-\Gz},\quad\Gz\in\bb{H}_{+}.
\end{equation}
An immediate corollary of (\ref{keyK}) is
\begin{lemma}
\label{lem:desnisty}
  $\CR(\CK)$ is dense in $L^{2}(\GG)$.
\end{lemma}
\begin{proof}
  Suppose $u\in L^{2}(\GG)$ is orthogonal to $\CR(\CK)$. Then for any $v\in
  L^{2}(\GG)$ we have $(\CK v,u)=0$. Choosing $v=u$ we obtain
\[
0=(u,\CK u)=\|\CK u\|^{2}_{H^{2}},
\]
which implies that $\CK u=0$. This implies that $u=0$ in $L^{2}(\GG)$. This
conclusion is obtained by observing that for any $v\in L^{2}(\GG)$ the image 
$f(\Gz)=(\CK v)(\Gz)$ has an analytic
extension to $\bb{C}\setminus\bra{\GG}$ and by the Sokhotski-Plemelj formula
\begin{equation}
  \label{SP}
  \jump{\CK v}_{\bra{\GG}}(s)=\frac{v(\tau(s))}{\bra{\dot{\tau}(s)}},
\end{equation}
where $\tau(s)$ is the arc-length parametrization of $\GG$.  For an oriented curve
$C\subset\bb{C}$ with parametrization $\tau(s)$ the notation $\jump{f}_{C}(s)$ means
\[
\jump{f}_{C}(s)=\lim_{\Gz\to\tau(s)^{+}}f(\Gz)-
\lim_{\Gz\to\tau(s)^{-}}f(\Gz),
\]
where $\Gz\to\tau(s)^{+}$ means that the vectors
$\dot{\tau}(s)$, $\Gz-\tau(s)$ form a positively oriented pair.

Thus, if $\CK u=0$ in $L^{2}(\GG)$ it follows that the unique analytic
extension of $\CK u$ is a zero function and (\ref{SP}) implies $u=0$.
\end{proof}
We remark that in the course of the proof of the Lemma we have also shown that
$\CK$ is a positive operator with trivial null-space.
We proceed now to the proof of (\ref{(K^-1 p_z, p_z) and pi_j}) by showing
that it is a consequence of a more general and elegant result about the
operator $\CK$ (see Lemma~\ref{lem:Hplus} below).

On a dense subspace $\CR(\CK)\subset L^{2}(\GG)$ we define a new
inner product
\[
(f,g)_{+}=(f,g)_{H^{2}}=(\CK u,v)_{L^{2}(\GG)},\qquad f=\CK u,\quad g=\CK v,
\]
where formula (\ref{keyK}) has been used. Suppose that $f_{n}=(f,e_{n})$ and
$g_{n}=(g,e_{n})$, then $f_{n}=\Gl_{n}u_{n}$ and $g_{n}=\Gl_{n}v_{n}$, where
$u_{n}=(u,e_{n})$ and $v_{n}=(v,e_{n})$. Then
\[
(f,g)_{+}=\sum_{n=1}^{\infty}\Gl_{n}u_{n}\bra{v_{n}}=\sum_{n=1}^{\infty}\frac{f_{n}\bra{g_{n}}}{\Gl_{n}}.
\]
We now define the Hilbert space $H_{+}$ as the completion of $\CR(\CK)$ with
respect to $\|\cdot\|_{+}$. Then
\begin{equation*}
H_+ = \left\{ f \in L^2(\Gamma) : \|f\|_+^2 := \sum_{n=1}^\infty \frac{|f_n|^2}{\lambda_n} < \infty \right\}
\end{equation*}
is a dense subspace in $H_{0}=L^2(\Gamma)$. In particular
$\|f\|^{2}_{+}\ge\Gl_{1}^{-1}\|f\|^{2}$.
\begin{lemma}
  \label{lem:Hplus}
$H_{+}$ consists of those functions in $L^{2}(\GG)$ that have (a necessarily
unique) extension to functions in $H^{2}(\bb{H}_{+})$. Moreover,
\begin{equation}
  \label{plusip}
  (f,g)_{+}=(f,g)_{H^{2}}.
\end{equation}
\end{lemma}
\begin{proof}
  Formula (\ref{plusip}) holds for all $\{f,g\}\subset\CR(\CK)$ by
  definition. Suppose that $f\in H_{+}$. We define
\[
\phi_{N}=\sum_{j=1}^{N}f_{n}e_{n}.
\]
Obviously, $\phi_{N}\in\CR(\CK)\subset H^{2}$, since each eigenfunction
$e_{n}$ is in $\CR(\CK)$. But then by (\ref{plusip}), $\phi_{N}$ would be a Cauchy sequence in
the $H^{2}$ norm and would have a limit $\phi_{\infty}\in H^{2}$. By
construction $\phi_{N}\to f$ in $H_{+}$. In particular $\phi_{N}\to f$ in
$L^{2}(\GG)$, but $\phi_{N}\to\phi_{\infty}$ in $H^{2}$ and therefore in
$L^{2}(\GG)$. Hence, $f=\phi_{\infty}$ on $\GG$ and $f$ has the extension
$\phi_{\infty}\in H^{2}$. Thus, if $\{f,g\}\subset H_{+}$ then $f$ and $g$
have extensions to $\bb{H}_{+}$ that are in $H^{2}(\bb{H}_{+})$. Moreover, if
\[
\psi_{N}=\sum_{j=1}^{N}g_{n}e_{n},
\]
then we can pass to the limit on both sides of the equality
\[
(\phi_{N},\psi_{N})_{+}=(\phi_{N},\psi_{N})_{H^{2}}
\]
and obtain (\ref{plusip}). To finish the proof we only need to show that
restrictions to $\GG$ of $H^{2}$ functions are in $H_{+}$. According to
(\ref{keyK})
\[
(e_{n},e_{m})_{H^{2}}=\nth{\Gl_{n}}(\CK
e_{n},e_{m})_{H^{2}}=\nth{\Gl_{n}}(e_{n},e_{m})=
\frac{\Gd_{mn}}{\Gl_{n}}.
\]
Hence the eigenbasis functions $e_{n}$ also form an orthogonal system in
$H_{+}$, but they are no longer orthonormal. We now take $f\in H^{2}$ and
repeat the proof of Bessel's inequality, using the orthogonality of $\{e_{n}\}$:
\[
0\le\|f-\sum_{n=1}^{N}f_{n}e_{n}\|_{H^{2}}^{2}=\|f\|_{H^{2}}^{2}
-2\sum_{n=1}^{N}\bra{f_{n}}(f,e_{n})_{H^{2}}+\sum_{n=1}^{N}\frac{|f_{n}|^{2}}{\Gl_{n}}=
\|f\|_{H^{2}}^{2}-\sum_{n=1}^{N}\frac{|f_{n}|^{2}}{\Gl_{n}},
\]
since, according to (\ref{keyK}),
\[
(f,e_{n})_{H^{2}}=\nth{\Gl_{n}}(f,\CK e_{n})_{H^{2}}=\frac{f_{n}}{\Gl_{n}}.
\]
Thus,
\[
\sum_{n=1}^{N}\frac{|f_{n}|^{2}}{\Gl_{n}}\le\|f\|_{H^{2}}^{2}
\]
and hence the series is convergent, proving that the restriction of $f\in
H^{2}$ to $\GG$ belongs to $H_{+}$. The Lemma is now proved.
\end{proof}
\begin{corollary}
  \begin{equation}
    \label{converge}
      \frac{\pi}{\Im z}=\|p_{z}\|_{H^{2}}^{2}=\|p_{z}\|_{+}^{2}=
\sum_{n=1}^{\infty}\frac{|\pi_{n}|^{2}}{\Gl_{n}},
  \end{equation}
establishing (\ref{(K^-1 p_z, p_z) and pi_j}) and hence \eqref{uofz}, which in
the new notation of inner product in $H_{+}$ can also be written as
\begin{equation}
  \label{uplus}
  2\pi u_{\epsilon, z}(z)=(u_{\epsilon, z},p_{z})_{+}=(u_{\epsilon, z},p_{z})_{H^{2}}.
\end{equation}
\end{corollary}
Lemma~\ref{upper bound sums THM} is now proved.
\end{proof} 
\begin{corollary}
\label{cor:Pyth}
By Lemma~\ref{uofz}
\[
2\pi u_{\epsilon, z}(z)=\sum_{n=1}^\infty\frac{|\pi_n|^2}{\lambda_n(\lambda_n+\epsilon^2)}=
\sum_{n=1}^\infty\frac{|\pi_n|^2}{(\lambda_n+\epsilon^2)^{2}}+
\Ge^{2}\sum_{n=1}^\infty\frac{|\pi_n|^2}{\Gl_{n}(\lambda_n+\epsilon^2)^{2}}
=\|u_{\epsilon, z}\|^{2}+\Ge^{2}\|u_{\epsilon, z}\|^{2}_{+},
\]
which in view of Lemma~\ref{lem:Hplus} proves
\begin{equation}
  \label{Pyth}
  2\pi u_{\epsilon, z}(z)=\|u_{\epsilon, z}\|_{L^2(\Gamma)}^2+\epsilon^{2}\|u_{\epsilon, z}\|_{H^2}^2.
\end{equation}
\end{corollary}

\begin{remark}
  \label{rem:Kinv}
For all $\{f,g\}\subset H_{+}$ we can formally write
\begin{equation}
  \label{Kinvpl}
  (f,g)_{+}=\sum_{n=1}^{\infty}\frac{f_{n}\bra{g_{n}}}{\Gl_{n}}=(\CK^{-1}f,g).
\end{equation}
If $f=\CK u$ for some $u\in L^{2}(\GG)$, the \rhs\ of (\ref{Kinvpl}) is equal
to $(u,g)$. Otherwise, $(f,g)_{+}$ will serve as a definition\footnote{The
  theory of rigged Hilbert spaces \cite{gesh64} can be used to
  define Hilbert space $H_{-}$ of generalized functions where $\CK^{-1}f$
  belongs for all $f\in H_{+}$. This space is naturally identified with the
  dual $(H_{+})^{*}$, so that $(\CK^{-1}f,g)$ is understood as the duality
  pairing between $\CK^{-1}f\in H_{-}=(H_{+})^{*}$ and $g\in H_{+}$. Most
  commonly this theory is used to define negative Sobolev spaces $W^{-m,2}$,
  where the role of $\CK^{-1}$ is played by an elliptic differential
  operator.} 
of $(\CK^{-1}f,g)$.
\end{remark}

\subsection{From exponential decay to power law} 
\label{SEC K from exp decay to power law}
In this section we prove Theorem~\ref{connections THM}. We begin by showing that the
first part of the theorem holds under substantially weaker assumptions than
(\ref{lambda_n and pi_n asymp}).
Here we suppress dependence on $\Ge$ and $z$ from the notation. So that
$u_{\Ge,z}$ is denoted simply by $u$. As before, $\Gl_{n}$ denote the
eigenvalues of $\CK$ and $\pi_{n}=(p_{z},e_{n})$ are coordinates of $p_{z}$ in
the eigenbasis of $\CK$.
\begin{lemma}
  \label{lem:BDconj}
  Let $u$ solve (\ref{Ku +eps^2 u = kz}).
  Assume that in addition to already proved inequality (\ref{Zolot}) there
  exist $\tilde{\rho}, \sigma, \tilde{\sigma} \in (0,1)$, so that
\begin{equation} \label{optimality relations}
\lambda_{n+1} \geq \tilde{\rho} \lambda_n,
\qquad
\displaystyle \frac{|\pi_{n+1}|^2}{\lambda_{n+1}} \leq \sigma
\frac{|\pi_n|^2}{\lambda_n}, \qquad 
\frac{|\pi_{n}|^2}{\lambda^{2}_{n}} \leq \tilde{\sigma}\frac{|\pi_{n+1}|^2}{\lambda_{n+1}^{2}},
\qquad \forall n \ \text{large}
\end{equation}
Then $\|u\|\simeq\Ge\|u\|_{H^{2}}$, i.e. Conjecture~\ref{conjecture upper
  bound} is true.
\end{lemma}
\begin{proof}
 We have 
\[
\|u\|_{H^{2}}^{2}=\sum_{n=1}^\infty \dfrac{|\pi_n|^2}{\lambda_n (\lambda_n+\epsilon^2)^2},
\qquad
\|u\|^{2} = \sum_{n=1}^\infty \dfrac{|\pi_n|^2}{(\lambda_n + \epsilon^2)^2}. 
\]
Define the switchover index $J=J(\epsilon)$
\begin{equation}
\label{switchJ}
\begin{cases}
\lambda_n \geq \epsilon^2 & \forall \ 1\leq n \leq J(\epsilon)
\\
\lambda_n < \epsilon^2 & \forall \ n> J(\epsilon)
\end{cases}
\end{equation}
 then we see
\begin{equation} \label{B(eps) bound}
\|u\|_{H^{2}}^{2} \simeq \sum_{n=1}^{J} \frac{|\pi_n|^2}{\lambda_n^3} + \frac{1}{\epsilon^4} \sum_{n=J + 1}^\infty \frac{|\pi_n|^2}{\lambda_n},\qquad
 \|u\|^{2}\simeq \sum_{n=1}^{J} \frac{|\pi_n|^2}{\lambda_n^2} + \frac{1}{\epsilon^4} \sum_{n=J + 1}^\infty |\pi_n|^2.
\end{equation}
Indeed, for $n\le J$
\[
\nth{4}\dfrac{|\pi_n|^2}{\lambda_n^{3}}\le\dfrac{|\pi_n|^2}{\lambda_n (\lambda_n+\epsilon^2)^2}\le\dfrac{|\pi_n|^2}{\lambda^{3}_{n}}
\]
while for all $n>J$
\[
\nth{4}\dfrac{|\pi_n|^2}{\lambda_n\Ge^{4}}\le\dfrac{|\pi_n|^2}{\lambda_n
  (\lambda_n+\epsilon^2)^2}\le\dfrac{|\pi_n|^2}{\lambda_{n}\Ge^{4}}
\]

\noindent recall that the eigenvalues are labeled in decreasing order: $\lambda_1 \geq \lambda_2 \geq ...$. Now, the second inequality of \eqref{optimality relations} implies

\begin{equation} \label{sum from J+1 to infty}
\frac{|\pi_n|^2}{\lambda_n} \leq \sigma^{n-J-1} \frac{|\pi_{J+1}|^2}{\lambda_{J+1}}, \qquad  n \geq J+1
\quad \Longrightarrow \quad
\sum_{n=J + 1}^\infty \frac{|\pi_n|^2}{\lambda_n} \leq \frac{1}{1-\sigma} \frac{|\pi_{J+1}|^2}{\lambda_{J+1}}
\end{equation}
But then we can estimate
\begin{equation*}
\|u\|_{H^{2}}^{2} \lesssim \frac{1}{\lambda_{J+1}}  \left( \sum_{n=1}^{J} \frac{\lambda_{J+1}}{\lambda_n^3} |\pi_n|^2 + \frac{|\pi_{J+1}|^2}{\epsilon^4} \right) 
\lesssim
\frac{1}{\lambda_{J+1}}  \left( \sum_{n=1}^{J} \frac{|\pi_n|^2}{\lambda_n^2}  + \frac{|\pi_{J+1}|^2}{\epsilon^4} \right)
\lesssim
\frac{\|u\|^{2}}{\Ge^{2}},
\end{equation*}
where in the last inequality we used the first inequality of \eqref{optimality relations}.

In order to prove the reverse inequality we appeal to the third inequality of
\eqref{optimality relations}, which implies
\begin{equation} \label{sum from 1 to J}
\sum_{n=1}^J \frac{|\pi_n|^2}{\lambda_n^2} \lesssim \frac{|\pi_{J}|^2}{\lambda_{J}^2} \le \lambda_J \sum_{n=1}^J \frac{|\pi_n|^2}{\lambda_n^3} 
\end{equation}
\noindent But then we can estimate
\begin{equation*}
\epsilon^2 \|u\|_{H^{2}}^{2} \gtrsim \frac{\epsilon^2}{\lambda_J} \left(\sum_{n=1}^J \frac{|\pi_n|^2}{\lambda_n^2} + \frac{1}{\epsilon^4} \sum_{n=J + 1}^\infty |\pi_n|^2\right)\gtrsim\|u\|^{2},
\end{equation*}
where we also used the first inequality of \eqref{optimality relations}:
$\lambda_J \lesssim \lambda_{J+1} < \epsilon^2$ which concludes the proof of
the lemma.
\end{proof}

Let us now prove the second part of Theorem~\ref{connections THM}, that
requires strict exponential asymptotics \eqref{lambda_n and pi_n asymp}, which
implies (\ref{optimality relations}), and therefore (\ref{optimality eqn}). In
this case formulas (\ref{main bound}), (\ref{Pyth}) and (\ref{L2normu}) imply
\[
M_{\Ge,z}\simeq\Ge^{2}\|u\|_{H^{2}}\simeq\Ge\|u\|=
\Ge\sqrt{\sum_{n=1}^\infty \frac{|\pi_n|^2}{(\Ge^{2}+\lambda_{n})^2}}.
\]
Then the conclusion of the second part of Theorem~\ref{connections THM}
follows from the following lemma.
\begin{lemma} \label{LEM sum asymptotics}
Let $\{a_n, b_n\}_{n=1}^\infty$ be nonnegative numbers such that $a_n \simeq e^{-\alpha n}$ and $b_n \simeq e^{-\beta n}$ with $0< \beta < \alpha$, where the implicit constants don't depend on $n$. Let $\Gd > 0$ be a small parameter, then

\begin{equation} \label{sum asymptotics}
\sum_{n=1}^\infty \frac{b_n}{(a_n + \Gd)^2} \simeq \Gd^{\frac{\beta}{\alpha} - 2}
\end{equation}

\noindent where the implicit constants don't depend on $\Gd$.

\end{lemma}

\begin{proof}
  As in the proof of Lemma~\ref{lem:BDconj} we introduce the switchover index
  $J=J(\Gd) \in \NN$ defined by

\begin{equation*}
\begin{cases}
a_n \geq \Gd & \quad \forall \ 1\leq n \leq J
\\
a_n < \Gd & \quad \forall \ n> J
\end{cases}
\end{equation*}

\noindent Below all the implicit constants in relations involving $\simeq$ or $\lesssim$ will be independent on $\Gd$. It is clear that

\begin{equation*}
\sum_{n=1}^\infty \frac{b_n}{(a_n + \Gd)^{2}} \simeq \sum_{n \leq J} \frac{b_n}{a_n^{2}} + \frac{1}{\Gd^{2}} \sum_{n>J} b_n.
\end{equation*}

\noindent Note that

\begin{equation*}
\sum_{n>J} b_n \lesssim \sum_{n>J} e^{-\beta n} \lesssim e^{-\beta (J+1)}.
\end{equation*}

\noindent On the one hand, using our assumption on $b_n$ we find

\begin{equation} \label{second sum}
\sum_{n>J} b_n \simeq b_{J+1} \simeq b_J.
\end{equation}

\noindent On the other hand

\begin{equation*}
\sum_{n \leq J} \frac{b_n}{a_n^{2}} \lesssim \sum_{n \leq J} e^{(2\alpha - \beta) n} = e^{2\alpha - \beta} 
\frac{e^{(2\alpha - \beta) J} - 1}{e^{2\alpha - \beta}-1} \lesssim e^{(2\alpha - \beta) J} \simeq \frac{b_J}{a_J^{2}},
\end{equation*}

\noindent Thus we conclude

\begin{equation} \label{first sum}
\sum_{n \leq J} \frac{b_n}{a_n^{2}} \simeq \frac{b_J}{a_J^{2}}.
\end{equation}

\noindent Now $\Gd \simeq a_J$ and $a_J \simeq e^{-\alpha J}$, therefore $e^{-J} \simeq \Gd^{\frac{1}{\alpha}}$. Using these along with \eqref{second sum} and \eqref{first sum} we obtain   

\begin{equation*}
\sum_{n=1}^\infty \frac{b_n}{(a_n + \Gd)^{2}} \simeq \frac{b_J}{a_J^{2}} + \frac{b_J}{\Gd^{2}} \simeq \frac{b_J}{a_J^{2}} \simeq e^{(2\alpha- \beta) J} \simeq \Gd^{\frac{\beta}{\alpha} - 2}.
\end{equation*}

\end{proof}

\section{Maximizing the extrapolation error} \label{DMP Section}
\setcounter{equation}{0} 
\label{sec:proof1}

\noindent \textbf{Notation:} In this section it will be convenient to switch
notation and let $\|\cdot\|$ and $(\cdot, \cdot)$ be the norm and the inner product of $H^2$.

\vspace{.1in}

Our goal is to understand how large $|f(z)|$ can be, under the assumptions
$\|f\|_{H^2} \leq 1$ and $\|f\|^2_{L^{2}(\GG)} \leq \epsilon$. From the representation formula \eqref{H2rep} we find

\begin{equation*}
f(z) = \frac{1}{2\pi} (f,p), \qquad \qquad p(x) = \frac{i}{x-\overline{z}}
\end{equation*}

\noindent on the other hand

\begin{equation*}
\|f\|_{L^2(\Gamma)}^2 = (\CK f,f)
\end{equation*}

\noindent where 

\begin{equation*}
(\CK f)(s) = \int_{\RR} k(t, s) f(t) \D t, \qquad
k(t,s) = \frac{1}{4\pi^2} \int_{\Gamma} \frac{|\D \tau|}{(t-\tau)(s-\overline{\tau})},\qquad
s\in\bb{R},
\end{equation*}
and we see that $\CK$ is a bounded, positive definite, self-adjoint operator in $H^2$. 
We can interchange the order of integration in the definition of $\CK f$, use \eqref{H2rep}, and obtain an alternative
representation:
\begin{equation}
  \label{K falt}
  (\CK f)(s) = \frac{i}{2\pi}\int_{\Gamma} \frac{f(\tau)|\D
    \tau|}{s-\overline{\tau}},\quad s\in\bb{R},\quad f\in H^2(\bb{H}_{+}).
\end{equation}
From this representation it is obvious that $\CK f$
has an analytic extension from $\bb{R}$ to $\mathbb{C} \setminus \overline{\Gamma}$ and that
its restriction to $\bb{H}_{+}$ is of Hardy class $H^{2}$.

\noindent Thus we arrive at a convex maximization problem with two quadratic constraints. Since the constraints are invariant with respect to the
choice of the constant phase factor for the function $f$, instead of
maximizing $|f(z)|$ we consider the
equivalent problem of maximizing a real linear functional $\Re (f,
p)$:

\begin{equation} \label{maximization problem}
\begin{cases}
\tfrac{1}{2\pi} \Re (f, p) \rightarrow \max \\
(f, f) \leq 1 
\\
\displaystyle (\CK f,f) \leq \epsilon^2 
\end{cases}
\end{equation}

\noindent For every $f$, satisfying \eqref{maximization problem}(b) and \eqref{maximization problem}(c) 
and for every
nonnegative numbers $\mu$ and $\nu$ ($\mu^{2}+\nu^{2}\not=0$) we have the
inequality
\[
((\mu + \nu \CK)f, f)\le\mu+\nu \epsilon^{2}
\]
obtained by multiplying \eqref{maximization problem}(b) by $\mu$ and \eqref{maximization problem}(c) by $\nu$
and adding. Also, for any uniformly positive definite self-adjoint operator
$M$ on $H^{2}$ we have
\[
\Re(u,v)-\frac{1}{2}(M^{-1}v,v)\le\frac{1}{2} (Mu,u)
\]
valid for all functions $u,v \in H^2$ (expand
$(M(M^{-1}v-u),(M^{-1}v-u))\geq 0$). The uniform positivity of $M$ ensures
that $M^{-1}$ is defined on all of $H^{2}$. This is an example of convex
duality (cf. \cite{ekte76}) applied to the convex function
$F(u)=(Mu,u)/2$. Then we also have for $\mu>0$

\begin{equation}\label{cxdual}
\Re(f, p)-\frac{1}{2} \left((\mu+\nu \CK)^{-1} p,p \right)\le
\frac{1}{2} \left( (\mu+\nu \CK)f,f \right),
\end{equation}
so that
\begin{equation} \label{maxub}
\Re(f, p) \le \frac{1}{2} \left((\mu+\nu \CK)^{-1} p,p \right)
+\frac{1}{2} \left( \mu+\nu \epsilon^{2}  \right)
\end{equation}

\noindent which is valid for every $f$, satisfying \eqref{maximization
  problem}(b) and \eqref{maximization problem}(c) and all $\mu > 0$, $\nu\ge
0$. In order for the bound to be optimal we must have equality in
(\ref{cxdual}), which holds if and only if
\[
p = (\mu + \nu \CK) f,
\]
giving the formula for optimal vector $f$:

\begin{equation} \label{maxxi}
  f=(\mu + \nu \CK)^{-1} p.
\end{equation}
The goal is to choose the Lagrange multipliers $\mu$ and $\nu$ so that the
constraints in (\ref{maximization problem}) are satisfied by $f$, given by
(\ref{maxxi}). Let us first consider special cases.

\noindent $\bullet$ if $\nu = 0$, then $f = \frac{p}{\|p\|} $, so we see that $f$ does not depend on the small parameter $\epsilon$, which leads to a contradiction, because the second constraint $(\CK f,f) \leq \epsilon^2$ is violated if $\epsilon$ is small enough. 

\vspace{.1in}

\noindent $\bullet$ If $\mu= 0$, the operator
$(\mu + \nu \CK)^{-1}$ is not defined on all of $H^{2}$. It is however defined
on a dense subspace of $H^{2}$. Even so, the choice $\mu=0$ cannot be optimal,
since then the optimal function $f$ would satisfy $\CK f = \frac{1}{\nu}p$. This
equation has no solutions in $H^2$, since $p$ has a pole at
$\overline{z}\not\in\overline{\Gamma}$, while $\CK f$ has an analytic extension to $\mathbb{C}
\setminus \overline{\Gamma}$.

\vspace*{.1in}

Thus we are looking for $\mu>0,\  \nu> 0$, so that equalities in
\eqref{maximization problem} hold. (These are the complementary slackness relations in Karush-Kuhn-Tucker conditions.), i.e.
\begin{equation}\label{cxdeq}
\begin{cases}
\left( (\mu + \nu \CK)^{-1} p, (\mu + \nu \CK)^{-1} p \right) =1, \\
\left( \CK (\mu + \nu \CK)^{-1} p, (\mu + \nu \CK)^{-1} p \right)= \epsilon^2.
\end{cases}
\end{equation}

\noindent Let $\eta = \frac{\mu}{\nu}$, we can solve either the first or the
second equation in \eqref{cxdeq} for $\nu$
\begin{equation}
  \label{nu1}
  \nu^{2} = \|(\CK+\eta)^{-1} p\|^{2},
\end{equation}
or
\begin{equation}
  \label{nu2}
  \nu^{2}=\epsilon^{-2}\left( \CK (\eta + \CK)^{-1} p, (\eta + \CK)^{-1} p \right).
\end{equation}
The two analysis paths stemming from using one or the other representation
for $\nu$ lead to two versions of the upper bound on $|f(z)|$, 
optimality of neither we can prove. However, the minimum of the two upper
bounds is still an upper bound and its optimality is then apparent. At first
glance both expressions for $\nu$ should be equivalent and not lead
to different bounds. Indeed, their
equivalence can be stated as an equation
\begin{equation}\label{Phi(eta) def and equation}
\Phi(\eta):= \frac{ \left( \CK (\CK + \eta)^{-1} p, (\CK + \eta)^{-1} p \right) }
{ \|(\CK+\eta)^{-1} p\|^2 } = \epsilon^{2}
\end{equation}
for $\eta$. We will prove that this equation has a unique solution
$\eta_{*}=\eta_{*}(\Ge)$, but we will be unable to prove that
$\eta_{*}(\Ge)\simeq\Ge^{2}$, as $\Ge\to 0$, which would follow from the
purported strict exponential decay of $\Gl_{n}$ and $\pi_{n}$. Thus, we take
$\eta_{*}(\Ge)=\Ge^{2}$ without justification, observing that \emph{any}
choice of $\eta$ gives a valid upper bound. But then the two expressions
(\ref{nu1}) and (\ref{nu2}) for $\nu$ give non-identical upper bounds, whose
combination achieves our goal.

We observe that
\[
\lim_{\eta\to\infty}\Phi(\eta)=\lim_{\eta\to\infty}\frac{ \left( \CK (\eta^{-1}\CK + 1)^{-1} p, (\eta^{-1}\CK + 1)^{-1} p \right) }
{ \|(\eta^{-1}\CK+1)^{-1} p\|^2 }=
\frac{(\CK p, p)}{\|p\|^{2}} < +\infty.
\]
Using Lemma~\ref{lem:Hplus} we have
\begin{equation}
  \label{Phinumer}
  \left( \CK (\CK + \eta)^{-1} p, (\CK + \eta)^{-1} p\right)=
\sum_{n=1}^{\infty}\frac{|\pi_{n}|^{2}}{(\Gl_{n}+\eta)^{2}},
\end{equation}
and
\begin{equation}
  \label{Phidenom}
  \|(\CK + \eta)^{-1} p\|^{2}=\sum_{n=1}^{\infty}\frac{|\pi_{n}|^{2}}{\Gl_{n}(\Gl_{n}+\eta)^{2}}.
\end{equation}
From Lemma Fatou and (\ref{diverge}) we know that
\[
\lim_{\eta\to 0}\|(\CK + \eta)^{-1} p\|^{2}=+\infty.
\]
Let $\Gd>0$ be arbitrary. Let $K$ be such that $\Gl_{n}<\Gd$ for all $n>
K$. Then
\[
\Phi(\eta)=\Phi_{K}(\eta)+\Psi_{K}(\eta),
\]
where
\[
\Phi_{K}(\eta)=\frac{\sum_{n=1}^{K}\frac{|\pi_{n}|^{2}}{(\Gl_{n}+\eta)^{2}}}
{\|(\CK + \eta)^{-1} p\|^{2}},\qquad
\Psi_{K}(\eta)=\frac{\sum_{n=K+1}^{\infty}\frac{|\pi_{n}|^{2}}{(\Gl_{n}+\eta)^{2}}}
{\|(\CK + \eta)^{-1} p\|^{2}}
\]
Then
\[
\lims_{\eta\to 0}\Phi_{K}(\eta)=0.
\]
We also have
\[
\Psi_{K}(\eta)\le\frac{\sum_{n=K+1}^{\infty}\frac{|\pi_{n}|^{2}}{(\Gl_{n}+\eta)^{2}}}
{\sum_{n=K+1}^{\infty}\frac{|\pi_{n}|^{2}}{\Gl_{n}(\Gl_{n}+\eta)^{2}}}\le
\Gl_{K+1}<\Gd.
\]
Thus,
\[
\lims_{\eta\to 0}\Phi(\eta)\le\lims_{\eta\to 0}\Phi_{K}(\eta)+
\lims_{\eta\to 0}\Psi_{K}(\eta)\le\Gd.
\]
Since $\Gd>0$ was arbitrary we conclude that
$\Phi(0^{+})=0$. Thus, for every $\Ge<\sqrt{(\CK p,p)}/\|p\|$ equation
(\ref{Phi(eta) def and equation}) has at least one solution $\eta>0$. We can prove that
this solution is unique by showing that $\Phi(\eta)$ is a monotone increasing
function. To prove this we only need to write
the numerator $N(\eta)$ of $\Phi'(\eta)$, obtained by the quotient rule. Using
formula
\[
\frac{d}{d\eta}(\CK+\eta)^{-1}=-(\CK+\eta)^{-2}
\]
and denoting $u=(\CK+\eta)^{-1}p$ we obtain
\[
N(\eta)=2((\CK+\eta)^{-1}u,u)(\CK u,u)-2(\CK(\CK+\eta)^{-1}u,u)\|u\|^{2}.
\]
Using formula $\CK (\CK+\eta)^{-1} = 1 - \eta (\CK+\eta)^{-1}$ we also have
\[
N(\eta)=2((\CK+\eta)^{-1}u,u)((\CK+\eta)u,u)-2(u,u)^{2}.
\]
Since operator $\CK+\eta$ is positive definite we can use the inequality
\[
(Ax,y)^{2}\le(Ax,x)(Ay,y)
\]
for $A=\CK+\eta$, $x=(\CK+\eta)^{-1}u$ and $y=u$, showing that $N(\eta)\ge 0$.
The equality occurs \IFF\ $x=\Gl y$. In our case this would correspond to $p$
being an eigenfunction of $\CK$, which is never true, since $p$ has a pole at
$\bar{z}$ and all functions in the range of $\CK$ have an analytic extension to 
$\mathbb{C}\setminus \overline{\Gamma}$. Thus, $N(\eta)>0$ and
(\ref{Phi(eta) def and equation}) has a unique
 solution $\eta_{*}>0$. Finding the asymptotics of $\eta_{*}(\Ge)$, as $\Ge\to
 0$ lies beyond capabilities of classical asymptotic methods because
 $\Phi(\eta)$ has an essential singularity at $\eta=0$. Indeed, it is not hard
 to show\footnote{Specifically $\eta=-\Gl_{m}$ is a pole of order 4 of
   $\|u\|^{4}$, while it is a pole of order 3 of $N(\eta)$.} that
 $\Phi'(-\Gl_{m})=0$ for all $m\ge 1$. Thus $\eta=0$ is neither a pole nor a
 removable singularity of $\Phi(\eta)$.

We can avoid the difficulty by observing that since the bound \eqref{maxub} is valid
for \emph{any} choice of $\mu$ and $\nu$, we can choose $\eta=\mu/\nu$ based on a
non-rigorous analysis of what $\eta_{*}$ should be, and then choose $\nu$ according to
(\ref{nu1}) or (\ref{nu2}), while still obtaining an
upper bound.

In accordance with (\ref{lambda_n and pi_n asymp}) we postulate that
\begin{equation*}
|\pi_n|^2 =  e^{- n \beta}, \qquad \Gl_n:= e^{-n \alpha}
\end{equation*}

\noindent for some $0< \alpha < \beta < 2 \alpha$, hence  equations \eqref{Phinumer} and \eqref{Phidenom} give

\begin{equation} \label{mPhi(eta) the model equation}
\Phi(\eta) = \frac{\displaystyle \sum_{n=1}^\infty \frac{e^{-n\beta}}{(e^{-n \alpha} + \eta )^2}}{\displaystyle \sum_{n=1}^\infty \frac{e^{n(\Ga- \beta)}}{(e^{-n \alpha} + \eta )^2}} = \epsilon^2.
\end{equation}
When $e^{-n \alpha} > \eta$ we will neglect $\eta$, while when
$e^{-n \alpha} < \eta$ we will neglect $e^{-n \alpha}$. Let $J = J(\eta)$ be
the switch-over index, for which $e^{-J(\eta)\Ga}\approx\eta$. Then 
\begin{equation*}
\begin{split}
\sum_{n=1}^\infty \frac{e^{n(\alpha-\beta)}}{(e^{-n \alpha} + \eta )^2} \approx \sum_{j=1}^J e^{n(3\alpha - \beta)} + \frac{1}{(\eta)^2} \sum_{j=J+1}^\infty e^{n(\alpha - \beta)}
\approx e^{J(3\alpha - \beta)} + \frac{e^{J(\alpha-\beta)}}{\eta^2}
\end{split}
\end{equation*}
Similarly,
\begin{equation*}
\sum_{n=1}^\infty \frac{e^{-n\beta}}{(e^{-n \alpha} + \eta )^2} \approx e^{J (2\alpha-\beta)} + \frac{e^{-J\beta}}{\eta^2}
\end{equation*}

\noindent substituting these approximations in \eqref{mPhi(eta) the model
  equation} and simplifying we obtain $e^{-J \alpha} \approx \epsilon^{2}$.
In other words 
\begin{equation} \label{eta = m epsilon^2}
\eta_{*}\approx e^{-J\Ga}\approx\epsilon^{2}.
\end{equation}
With this motivation let us choose $\eta = \epsilon^2$. With this and formulas
(\ref{nu1}) and (\ref{nu2}) for $\nu$ we obtain the two forms of the upper
bound (\ref{maxub}) conveniently written in terms of $u=(\CK+\Ge^{2})^{-1}p$:
\begin{equation}
  \label{UB1}
\Re(f,p)\le \frac{(u,p)}{2\|u\|}+\Ge^{2}\|u\|=\frac{\pi u(z)}{\|u\|}+\Ge^{2}\|u\|,
\end{equation}
where we have used (\ref{uplus}), and similarly
\begin{equation}
  \label{UB2}
\Re(f,p)\le \Ge\frac{\pi u(z)}{\|u\|_{L^{2}(\GG)}}+\Ge\|u\|_{L^{2}(\GG)},
\end{equation}
By (\ref{Pyth})
\[
\Ge^{2}\|u\|\le\frac{2\pi u(z)}{\|u\|},\qquad
\|u\|_{L^{2}(\GG)}\le\frac{2\pi u(z)}{\|u\|_{L^{2}(\GG)}}.
\]
Therefore, we have both
\[
|f(z)|=\nth{2\pi}\Re(f,p)\le\frac{3}{2}\frac{u(z)}{\|u\|},\qquad
|f(z)|\le\frac{3\Ge}{2}\frac{u(z)}{\|u\|_{L^{2}(\GG)}}.
\]
Inequality (\ref{main bound}) is now proved. 

We remark that equation (\ref{Phi(eta) def and equation}) for the optimal
choice $\eta_{*}(\Ge)$ can be written as
$\|u\|_{L^{2}(\GG)}=\Ge\|u\|_{H^{2}}$, in which case solution of (\ref{Ku
  +eps^2 u = kz}) with $\eta_{*}(\Ge)$ in place of $\Ge$ would satisfy
\[
\frac{\Ge u_{\Ge,z}(z)}{\|u\|_{L^{2}(\GG)}}=\frac{u_{\Ge,z}(z)}{\|u\|_{H^{2}}}=M_{\Ge,z}.
\]
Moreover $M_{\Ge,z}$ would be an exact upper bound for $|f(z)|$ achieved by
both $\Ge u_{\Ge,z}(\Gz)/\|u\|_{L^{2}(\GG)}$ and
$u_{\Ge,z}(\Gz)/\|u\|_{H^{2}}$. In the absence of exact asymptotics of
$\eta_{*}(\Ge)$ we have obtained only a marginally weaker bound, differing
from the optimal by at most a small constant multiplicative factor.

\section{Proof of Theorem~\ref{main boundary THM}} \label{extrapolation from boundary SECT}
\setcounter{equation}{0}

\subsection{The integral equation}

Let us first establish an analogous result to Theorem~\ref{upper bound THM}, i.e. below we formulate the upper bound in the case $\Gamma = [-1,1]$ via the solution to an integral equation.

\begin{theorem}\label{upper bound boundary THM}
Let $z \in \HH_+$ and $\epsilon > 0$. Assume $f \in H^2$ is such that $\|f\|_{H^2} \leq 1$ and $\|f\|_{L^2(-1,1)} \leq \epsilon$, then

\begin{equation} \label{main bound boundary}
|f(z)| \leq \frac{3}{2} \epsilon \frac{u_{\epsilon, z}(z)}{\|u_{\epsilon, z}\|_{L^2(-1,1)}}
\end{equation}

\noindent where $u_{\epsilon, z}$ solves the integral equation

\begin{equation} \label{Ku +(1+eps^2) u = p}
\hf(K u + u)+\epsilon^2 u = p_z, \qquad \text{on} \ (-1,1)
\end{equation}

\noindent with

\begin{equation} \label{K truncated HT}
K u (x) = \frac{i}{\pi} \int_{-1}^1 \frac{ u(y)}{x - y} \D y , \qquad \qquad p_z(\xi) = \frac{i}{\xi- \overline{z}},
\end{equation}
where the integral is understood in the principal value sense.
\end{theorem}  
\begin{proof}
It is enough to prove the inequality \eqref{main bound boundary} for $\|f\|_{H^2} \leq 1$ and $\|f\|_{L^2(-1,1)} < \epsilon$, because when $\|f\|_{L^2(-1,1)} = \epsilon$ we can consider the sequence $f^n:=(1-\frac{1}{n}) f$ and take limits in the inequality for $f^n$ as $n \to \infty$.

Since $f(\cdot+ih) \to f$ as $h \downarrow 0$ in $L^2(-1,1)$ (a well-known property of $H^2$ functions, see \cite{koos98}), the assumption $\|f\|_{L^2(-1,1)} < \epsilon$ implies that $\|f(\cdot + ih)\|_{L^2(-1,1)} \leq \epsilon$ for $h$ small enough. In other words $\|f\|_{L^2(\Gamma_h)} \leq \epsilon$, where $\Gamma_h = [-1,1] + ih$, so we can apply Theorem~\ref{upper bound THM} and conclude

\begin{equation*}
|f(z)| \leq \frac{3}{2} \epsilon \frac{u_h(z)}{\|u_h\|_{L^2(\Gamma_h)}}, \qquad \qquad \forall h \ \text{small enough}
\end{equation*}

\noindent where $u_h$ solves the integral equation

\begin{equation*}
\CK u(\zeta) +\epsilon^2 u(\zeta) = \frac{i}{\zeta- \overline{z}}, \qquad \qquad \zeta \in \Gamma_h
\end{equation*}

\noindent Let us set $v(x)=u(x+ih)$, then the above integral equation can be rewritten as

\begin{equation} \label{integral eqn for v}
\CK_h v(x) + \epsilon^2 v(x) = p_h(x), \qquad \qquad x \in [-1,1]
\end{equation}

\noindent with

\begin{equation} \label{K_h}
\CK_h v (x) = \frac{1}{2\pi} \int_{-1}^1 \frac{i v(y) \D y}{x - y + 2ih},
\qquad \qquad
p_h(x) = \frac{i}{x+ih- \overline{z}}
\end{equation}

\noindent again $\CK_h$ is a positive operator on $L^2(-1,1)$, $\CK_hv$ has analytic extension to the upper half-plane hence the solution $v$ of \eqref{integral eqn for v} is also analytic in $\HH_+$. Let us denote this solution by $v_h$ to indicate its dependence on the small parameter $h$, namely $v_h = \left( \CK_h + \epsilon^2 \right)^{-1} p_h$. Then the upper bound on $f$ becomes

\begin{equation} \label{upper bound on f with v_h}
|f(z)| \leq \frac{3}{2} \epsilon \frac{v_h(z-ih)}{\|v_h\|_{L^2(-1,1)}},\qquad \forall h \ \text{small enough}
\end{equation}

\noindent Our goal is to take limits in this upper bound as $h \downarrow 0$. 
\begin{lemma}
  \label{lem:Kh20}
Let $\CK_h$ and $K$ be defined by \eqref{K_h} and \eqref{K truncated HT}, respectively. Then any $g \in L^2(-1,1)$
\begin{equation}\label{K_h to K+1}
\CK_h g \to \tfrac{1}{2} (K+1)g, \qquad \qquad \text{as} \ h \downarrow 0, \ \text{in} \ L^2(-1,1).
\end{equation}
\end{lemma}
\begin{proof}~\\
  \noindent $\bullet$ $\{\CK_h\}_{h>0}$ is uniformly bounded in the operator norm on $L^2(-1,1)$. To prove this we observe that $\CK_h g = k * \chi_1 g $, where $\chi_1:=\chi_{(-1,1)}$ and

\begin{equation*}
k(t) = \frac{i}{2\pi(t+2ih)}
\end{equation*}

\noindent with the definition $\widehat{f}(\zeta) = \int_{\RR} f(x) e^{-i\zeta
  x} \D x$ we can compute $\widehat{k}(\zeta) = e^{-2h \zeta}
\chi_{>0}(\zeta)$, where $\chi_{>0}(\zeta)=\chi_{(0,+\infty)}(\zeta)$. In
particular $|\widehat{k}| \leq 1$, but then

\begin{equation*}
\begin{split}
\|\CK_h g\|_{L^2(-1,1)} &\leq \|\CK_h g\|_{L^2(\RR)} = \tfrac{1}{\sqrt{2\pi}} \|\widehat{k} \cdot \widehat{\chi_1 g}\|_{L^2(\RR)}
\leq \tfrac{1}{\sqrt{2\pi}} \|\widehat{\chi_1 g}\|_{L^2(\RR)} =
\\
&=  \|\chi_1 g\|_{L^2(\RR)}  = \|g\|_{L^2(-1,1)}
\end{split}
\end{equation*}

\noindent which immediately implies $\|\CK_h\| \leq 1$  for any $h>0$.

\vspace{.1in}

\noindent $\bullet$ By uniform boundedness of $\|\CK_{h}\|$, it is enough to
show convergence $\CK_h g \to \tfrac{1}{2} (K+1)g$ in $L^2(-1,1)$ for a dense
set of functions $g$. We will now show convergence for all $g \in
C_0^\infty(-1,1)$. Since by Sokhotski-Plemelj formula this convergence holds
a.e. in $(-1,1)$, to achieve the desired conclusion it is enough to show that
the family of functions $|\CK_h g|^2$ is equiintegrable in $(-1,1)$. Vitali
convergence theorem \cite[p. 133, exercise 10(b)]{rudin87} then implies
convergence of $\CK_h g$ in $L^2(-1,1)$. We recall the definition of
equiintegrability:
\begin{equation} \label{equi}
\sup_{|A| \leq \delta} \sup_{h>0} \int_{A} |\CK_h g(x)|^2 \D x \to 0, \qquad \qquad \text{as} \ \delta \to 0,
\end{equation}
where the first supremum is taken over measurable subsets $A \subset (-1,1)$.
\noindent We compute

\begin{equation*}
\int_{A} |\CK_h g(x)|^2 \D x = \|\chi_A \CK_h g\|_{L^2(\RR)}^2 = \|\widehat{\chi_A} * \widehat{\CK_h g}\|_{L^2(\RR)}^2 
\leq \|\widehat{\chi_A}\|_{L^2(\RR)}^2 \|\widehat{\CK_h g}\|_{L^1(\RR)}^2 
\end{equation*}

\noindent where we have used Young's inequality. Now \eqref{equi} follows from uniform boundedness of $\|\widehat{\CK_h g}\|_{L^1(\RR)}$. We compute

\begin{equation*}
\widehat{\CK_h g}(\xi) = e^{-2h\xi} \chi_{>0}(\xi) \widehat{\chi_1 g}(\xi)
\end{equation*}

\noindent hence

\begin{equation*}
\|\widehat{\CK_h g}\|_{L^1(\RR)} \leq \|\widehat{\chi_1 g}\|_{L^1(\RR)} = \|\widehat{g}\|_{L^1(\RR)} < \infty
\end{equation*}

\noindent since for $g \in C_0^\infty(-1,1)$ we have $\widehat{\chi_1 g} = \widehat{g} \in L^1(\RR)$. Thus,

\begin{equation*}
\int_{A} |\CK_h g(x)|^2 \D x \leq  \|\widehat{\chi_A}\|_{L^2(\RR)}^2 \|\widehat{g}\|_{L^1(\RR)}^2 = |A| \|\widehat{g}\|_{L^1(\RR)}^2 \to 0, \qquad \qquad \text{as} \ \delta \to 0
\end{equation*}
\end{proof}

Since $\CK_h$ is a positive operator for any $h$, we see that so is $K+1$ and hence the inverse of $\tfrac{1}{2} (K+1) + \epsilon^2$ is well-defined on $L^2(-1,1)$. We now see that, as $h \downarrow 0$

\begin{equation}
\label{Khconv}
v_h = \left( \CK_h + \epsilon^2 \right)^{-1} p_h \longrightarrow \left( \tfrac{1}{2} (K+1) + \epsilon^2 \right)^{-1} p =:w, \qquad \qquad \text{in} \ L^2(-1,1)
\end{equation}

\noindent where $p(x) = \frac{i}{x-\overline{z}}$. Using the resolvent identity
\[
\left( \CK_h + \epsilon^2 \right)^{-1}-\left( \CK_0 + \epsilon^2 \right)^{-1}=
\left( \CK_h + \epsilon^2 \right)^{-1}(\CK_{0}-\CK_{h})\left( \CK_0 + \epsilon^2
\right)^{-1},
\]
where $\CK_{0}=\frac{1}{2}(K+1)$, we conclude that
\[
\left( \CK_h + \epsilon^2 \right)^{-1}g\to\left( \CK_0 + \epsilon^2 \right)^{-1}g
\]
for any $g\in L^{2}(-1,1)$, since all operators above are uniformly
bounded as $h\to 0$. Relation (\ref{Khconv}) then easily follows.

We now observe that because of the convergence \eqref{K_h to K+1} $w \in L^2(-1,1)$ represents boundary values of an analytic function in the upper half-plane (in fact an $H^2$ function), hence we can extend $w$ to $\HH_+$, more specifically

\begin{equation*}
\epsilon^2 w(\zeta):= p(\zeta) - \frac{i}{2\pi} \int_{-1}^1 \frac{w(y)}{\zeta-y} \D y, \qquad \qquad \zeta \in \HH_+ 
\end{equation*}

\noindent defines the extension. But then, from the integral equation for $v_h$ we see that

\begin{equation*}
\epsilon^2 v_h(z-ih) = \frac{i}{z-\overline{z}} - \frac{i}{2\pi} \int_{-1}^1 \frac{v_h(y)}{z-y+ih} \D y \longrightarrow \epsilon^2 w(z)
\end{equation*}

\noindent and thus we conclude

\begin{equation*}
|f(z)| \leq \frac{3}{2} \epsilon \frac{w(z)}{\|w\|_{L^2(-1,1)}}
\end{equation*}

\noindent It remains to relabel $w$ by $u_{\epsilon, z}$ and conclude the proof.  
\end{proof}

\subsection{Solution of the integral equation}

The goal of this section is to find the function $u$ appearing in the upper bound \eqref{main bound boundary}. Recall that $u$ solves the integral equation

\begin{equation*}
K u +\lambda u = 2p, \qquad \text{on} \ (-1,1)
\end{equation*}

\noindent where $\lambda = 1+2\epsilon^2$, $K$ is the truncated Hilbert transform given by \eqref{K truncated HT}, and for fixed $z \in \HH_+$

\begin{equation*}
p(x) = \frac{i}{x - \overline{z}}
\end{equation*}

\noindent The reason that makes it possible to solve this integral equation, is the spectral representation of $K$ obtained in \cite{kopi59}. Below we state the results of \cite{kopi59}. For $x, \zeta \in (-1,1)$ let

\begin{equation}
\sigma(x, \zeta) = \frac{\text{exp} \left\{ \tfrac{i}{2\pi} L(x) L(\zeta) \right\}}{\pi \sqrt{(1-x^2)(1-\zeta^2)}} , \qquad \qquad L(x) = \ln\left( \frac{1+x}{1-x}\right)
\end{equation}

\begin{theorem}
The formulae

\begin{equation*}
\begin{split}
f(x) =  \int_{-1}^1 g(\zeta) \sigma(x, \zeta) \D \zeta,
\qquad \qquad
g(\zeta) =  \int_{-1}^1 f(x) \overline{\sigma(x, \zeta)} \D x
\end{split}
\end{equation*}

\noindent are inversion formulae which represent isometries from the space $L^2(-1,1)$ to itself. 
\end{theorem}

\begin{theorem}
If $f(x)$ corresponds to $g(\zeta)$, then $Kf(x)$ corresponds to $\zeta g(\zeta)$ (w.r.t. the above transformation).
\end{theorem}

\begin{remark}
Integrals are understood in a limiting sense as the Fourier transform of an $L^2$ function, namely as the limit of $\int_{-1+\delta}^{1-\delta}$ when $\delta \downarrow 0$ in the sense of $L^2(-1,1)$.
\end{remark}

\noindent Let $(\cdot, \cdot)$ denote the inner product of $L^2(-1,1)$, using the stated result we can write

\begin{equation*}
u(x) = \int_{-1}^1 \left(u, \sigma(\cdot, \zeta) \right) \sigma(x, \zeta) \D \zeta, \qquad \qquad 
p(x) = \int_{-1}^1 \left(p, \sigma(\cdot, \zeta) \right) \sigma(x, \zeta) \D \zeta
\end{equation*}

\begin{equation*}
Ku(x) = \int_{-1}^1 \zeta \left(u, \sigma(\cdot, \zeta) \right) \sigma(x, \zeta) \D \zeta
\end{equation*}

\noindent then the integral equation gives

\begin{equation*}
(\lambda + \zeta) \left(u, \sigma(\cdot, \zeta) \right) = 2 \left(p, \sigma(\cdot, \zeta) \right)
\end{equation*}

\noindent and therefore

\begin{equation} \label{u kopelman pincus}
u(x) = \int_{-1}^1 \frac{2\left(p, \sigma(\cdot, \zeta) \right) \sigma (x,\zeta)}{\lambda + \zeta} \D \zeta
\end{equation}

\noindent Let us compute $\left(p, \sigma(\cdot, \zeta) \right)$ explicitly by changing variables $y = \tanh(t)$, in which case $L(y) = 2t$. We obtain

\begin{equation*}
\left(p, \sigma(\cdot, \zeta) \right) = \frac{i}{\pi \sqrt{1-\zeta^2}} \int_{\RR} \frac{e^{-i t L(\zeta)/\pi}}{\sinh t - \overline{z} \cosh t} \D t
\end{equation*}

\noindent let $\alpha \in \bb{C}$ be such that $\coth \alpha = \overline{z}$, then

\begin{equation*}
\left(p, \sigma(\cdot, \zeta) \right) = -\frac{i \sinh \alpha}{\pi \sqrt{1-\zeta^2}} \int_{\RR} \frac{e^{-i t L(\zeta)/\pi}}{\cosh(t-\alpha)} \D t
\end{equation*}

\noindent We observe that

\begin{equation*}
\coth \alpha = \frac{e^{2\alpha} +1}{e^{2\alpha}-1} = \frac{w+1}{w-1}, \qquad \qquad w = e^{2\alpha}
\end{equation*}

\noindent The fractional linear map $w \mapsto \frac{w+1}{w-1}$ maps lower half-plane into the upper half-plane and therefore, $w = w(\overline{z})$ is in the upper half-plane. Hence, $\Im \alpha \in (0, \pi / 2)$. It follows that there are no zeros of $\cosh (t-\alpha)$ in the strip bounded by $\RR$ and $\Im t = \Im \alpha$. Taking into account that

\begin{equation*}
\lim_{R \to \infty} \int_0^{\Im \alpha} \frac{e^{-i(i\tau \pm R) L(\zeta) / \pi}}{\cosh (i\tau \pm R - \alpha)} i \D \tau = 0
\end{equation*}

\noindent we conclude that

\begin{equation*}
\left(p, \sigma(\cdot, \zeta) \right) = -\frac{i e^{-i \alpha L(\zeta)/\pi} \sinh \alpha}{\pi \sqrt{1-\zeta^2}} \int_{\RR} \frac{e^{-i t L(\zeta)/\pi}}{\cosh(t)} \D t
=
-\frac{i e^{-i \alpha L(\zeta)/\pi} \sinh \alpha}{\sqrt{1-\zeta^2} \cosh(L(\zeta) / 2)}
\end{equation*}

\noindent simplifying we obtain

\begin{equation*}
\left(p, \sigma(\cdot, \zeta) \right)
=
-i e^{-i \alpha L(\zeta)/\pi} \sinh \alpha
\end{equation*}

\noindent We now use this formula in \eqref{u kopelman pincus}.

\begin{equation*}
u(x) = -\frac{2i \sinh \alpha}{\pi \sqrt{1-x^2}} \int_{-1}^1 \frac{e^{i L(\zeta)[L(x)-2\alpha] / 2\pi}}{(\lambda + \zeta) \sqrt{1-\zeta^2}} \D \zeta
\end{equation*}

\noindent once again changing the variables $\zeta = \tanh s$ we obtain

\begin{equation*}
u(x) = -\frac{2i \sinh \alpha}{\pi \sqrt{1-x^2}} \int_\RR \frac{e^{i s[L(x)-2\alpha] / \pi}}{\sinh s + \lambda \cosh s} \D s
\end{equation*}

\noindent Let $\beta = \beta(\lambda)$ be such that $\coth \beta = \lambda$, then $\beta(\lambda) > 0$ and $\beta(\lambda) \to + \infty$, as $\lambda \to 1$. Now

\begin{equation*}
u(x) = -\frac{2i \sinh \alpha \sinh \beta}{\pi \sqrt{1-x^2}} \int_\RR \frac{e^{i s[L(x)-2\alpha] / \pi}}{\cosh (s+\beta)} \D s 
=
-\frac{2i \sinh \alpha \sinh \beta}{\sqrt{1-x^2}} \frac{e^{-i \beta [L(x)-2\alpha]/\pi}}{\cosh(L(x)/2 - \alpha)} 
\end{equation*}

\noindent Next we simplify

\begin{equation*}
\cosh\left( \tfrac{L(x)}{2} - \alpha \right) = \cosh\left( \tfrac{L(x)}{2} \right) \cosh \alpha - \sinh \left( \tfrac{L(x)}{2}\right) \sinh \alpha = \frac{\cosh \alpha - x \sinh \alpha}{\sqrt{1-x^2}}
\end{equation*}

\noindent Thus we obtain the final answer

\begin{equation} \label{u final answer}
u(x) = \frac{2i \sinh \beta}{x-\overline{z}} e^{-i \frac{\beta}{\pi} [L(x)-2\alpha]} = 2p(x) \sinh (\beta) e^{-i \frac{\beta}{\pi} [L(x)-2\alpha]}
\end{equation}

\noindent where (with $\ln$ denoting the principal branch of logarithm)

\begin{equation*}
\beta = \tfrac{1}{2} \ln \left(1 + \epsilon^{-2} \right),
\qquad \qquad \alpha = \tfrac{1}{2} \ln \frac{\overline{z}+1}{\overline{z} - 1}
\end{equation*}

\noindent We see that

\begin{equation*}
\|u\|_{L^2(-1,1)} = 2 \|p\|_{L^2(-1,1)} \sinh(\beta) e^{-2\frac{\beta}{\pi} \Im \alpha}
\end{equation*}

\noindent Because $\Re L(z) = 2 \Re \alpha$ and $e^\beta \sim \epsilon^{-1}$ as $\epsilon \to 0$, we find that

\begin{equation} \label{upper bound asymp boundary case}
\epsilon\frac{u(z)}{\|u\|_{L^2(-1,1)}} = \epsilon\frac{p(z) e^{\frac{\beta}{\pi} \Im L(z)}}{\|p\|_{L^2(-1,1)}} \sim \frac{p(z) \epsilon^{\frac{1}{\pi} [\pi - \Im L(z)]}}{\|p\|_{L^2(-1,1)}} =:B, \qquad \qquad \text{as} \ \epsilon \to 0
\end{equation}

\noindent Since $\frac{1+z}{1-z} \in \HH_+ $ we see that $\pi - \arg \frac{1+z}{1-z} = - \arg \frac{z+1}{z-1}$ and with $z = z_r + i z_i$ we obtain

\begin{equation}
B = \frac{\epsilon^{-\frac{1}{\pi} \arg \frac{z+1}{z-1}}}{2\sqrt{z_i} \sqrt{\arctan \frac{z_r + 1}{z_i} - \arctan \frac{z_r - 1}{z_i}}}
\end{equation}

\noindent when $\epsilon < 1$ we can replace the asymptotic equivalence to $B$ in \eqref{upper bound asymp boundary case} by $\leq \sqrt{2} B$ and conclude the proof of \eqref{main bound boundary explicit}. To prove the optimality of this upper bound we consider the function   

\begin{equation*}
W(\zeta) = \epsilon \frac{p(\zeta)}{\|p\|_{L^2(-1,1)}} e^{\frac{i \ln \epsilon}{\pi} \ln \frac{1+\zeta}{1-\zeta}}, \qquad \qquad \zeta \in \HH_+
\end{equation*}

\noindent clearly this is an analytic function in the upper half-plane and belongs to $H^2$, $\|W\|_{L^2(-1,1)} = \epsilon$ and 

\begin{equation*}
\|W\|_{H^2}^2 = \epsilon^2 + \frac{\|p\|_{L^2((-1,1)^C)}^2}{\|p\|_{L^2(-1,1)}^2} = \epsilon^2 - 1 + \frac{\pi}{\arctan \frac{z_r+1}{z_i} - \arctan \frac{z_r-1}{z_i}} \leq C
\end{equation*}

\noindent where $C>0$ is independent of $\epsilon$, therefore $W$ is an admissible function. Further,

\begin{equation*}
|W(z)| = B
\end{equation*}

\noindent that is, $W(\zeta)$ attains the bound \eqref{main bound boundary explicit} up to a constant independent of $\epsilon$.

\vspace{.2in}

\medskip

\noindent\textbf{Acknowledgments.}
The authors are grateful for the hospitality of Courant Institute, where part
of the work was done, while YG was a visiting member in the Spring 2018
semester. We have greatly benefited from long discussions with Percy Deift and
Bob Kohn about different approaches to the subject. Leslie Greengard provided
the quadruple precision FORTRAN code that enabled us to probe this otherwise
very ill-conditioned problem numerically. The authors thank Georg Stadler for
suggestions of related work on ill-posed problems. The authors also wish to
thank Alex Townsend for sharing his insights during his visit to Temple
University. Last, but not least, the authors are indebted to Mihai Putinar for
a lot of enlightening discussions about asymptotics of eigenvalues of integral
operators arising in analytic continuation problems, and for directing us to a
large trove of relevant literature. This material is based upon work supported
by the National Science Foundation under Grant No. DMS-1714287.


\bibliographystyle{abbrv}
\bibliography{refs}
\end{document}